\documentclass[11pt]{amsart}
\usepackage[utf8]{inputenc}
\usepackage[T1]{fontenc}
\usepackage{lmodern}
\usepackage{microtype,needspace}
\usepackage[margin=1in]{geometry} 
\usepackage[english]{babel}

\usepackage{amsmath, amsthm, amssymb, mathtools, mathrsfs, bbm}

\usepackage[ruled,vlined]{algorithm2e}
\usepackage{algorithmic}
\algsetup{indent=7mm}

\usepackage{graphicx}
\usepackage{epsfig} 
\usepackage{epstopdf} 

\usepackage{booktabs}
\usepackage{multirow}
\usepackage{longtable}

\usepackage{enumerate}

\usepackage{listings}
\usepackage{xcolor}
\usepackage{cite}
\usepackage{thmtools}
\usepackage{thm-restate}

\usepackage[pdftex, colorlinks=true, linkcolor=black, citecolor=black, urlcolor=black, pagebackref=false, pdfborder={0 0 0}]{hyperref}
\hypersetup{
  pdftitle={A PDE Perspective on Generative Diffusion Models},
  pdfauthor={Kang Liu and Enrique Zuazua},
  pdfsubject={Final revised manuscript, version 19},
  pdfkeywords={Fokker-Planck equation, score-based diffusion, Wasserstein convergence}
}
\usepackage{url}
\numberwithin{equation}{section}


\usepackage[all]{xy}
\usepackage{tikz-cd}

\usepackage{float}
\usepackage{caption}
\usepackage{subcaption}

\usepackage{relsize}
\usepackage[foot]{amsaddr} 
\usepackage{verbatim}  

\newcommand{\R}{\mathbb{R}}
\DeclareMathOperator*{\argmin}{argmin}

\newtheorem{thm}{Theorem}[section]
\newtheorem{prop}[thm]{Proposition}
\newtheorem{lem}[thm]{Lemma}

\theoremstyle{definition}

\newtheorem{rem}[thm]{Remark}

\theoremstyle{remark}

\numberwithin{equation}{section}

\date{}

\begin{document}
\title{A PDE Perspective on Generative Diffusion Models}

\author[K. Liu and E. Zuazua]{Kang Liu$^1$}
\address{$^1$Universit\'e Bourgogne Europe, CNRS, Institut de Math\'ematiques de Bourgogne, 21000 Dijon, France.}

\author{Enrique Zuazua$^2$$^3$$^4$}
\address{$^2$Friedrich\ -\ Alexander\ -\ Universit\"at Erlangen\ -\ N\"urnberg, Department of Mathematics, Chair for Dynamics, Control, Machine Learning, and Numerics (Alexander von Humboldt Professorship), 91058 Erlangen, Germany.}
\address{$^3$Universidad Aut\'onoma de Madrid, Departamento de Matem\'aticas, 28049 Madrid, Spain.}
\address{$^4$Chair of Computational Mathematics, Fundaci\'on Deusto, 48007 Bilbao, Basque Country, Spain.}
\email{kang.liu@u-bourgogne.fr, enrique.zuazua@fau.de}

\thanks{E. Zuazua was partially supported by the European Research Council (ERC) under the European Union's Horizon Europe research and innovation programme (grant agreement No.~101096251-CoDeFeL); by the Alexander von Humboldt Professorship program; the European Union's Horizon Europe MSCA project ModConFlex (HORIZON-MSCA-2021-DN-01, project 101073558); the Transregio 154 Project ``Mathematical Modelling, Simulation and Optimization Using the Example of Gas Networks'' of the DFG; the AFOSR 24IOE027 project; the SURE-AI Norwegian Centre for Sustainable, Risk-Averse, and Ethical AI grant 357482, Research Council of Norway; by the Grant PID2023-146872OB-I00-DyCMaMod of MICIU (Spain) and by the COST Actions CA24122 -- Multiscale Stochastics, Patterns, and Analysis of Combinatorial Environments and CA24136 -- Interactions between Control Theory and Machine Learning.}



\begin{abstract}
Score-based diffusion models have emerged as a powerful class of generative
methods, with successful applications across diverse domains.
Despite their empirical success, their mathematical foundations remain
only partially understood, particularly regarding the stability and
consistency of the stochastic and partial differential equations
underlying their dynamics.

{
In this work, we develop a partial differential equation (PDE) framework
for score-based diffusion processes associated with the heat flow of a
compactly supported data measure. After establishing weak well-posedness
on positive-time intervals, we derive a sharp one-sided
divergence bound for the score and use it to obtain sharp uniform
$L^p$-stability estimates for the generative
Fokker--Planck equation. This bound is the Euclidean form
of the Li--Yau differential inequality. Through entropy stability methods, we prove
that the generated laws converge in the $2$-Wasserstein distance at
rate $\sqrt t$ as $t\to0^+$ to a probability measure supported on the
data support. This holds for both deterministic and stochastic
generation, for positive terminal densities with finite absolute
entropy and second moment. In the deterministic case, almost every
trajectory converges to a point of the data support; for empirical
data, we further obtain a trajectory-wise square-root rate.

These results clarify how exact-score dynamics concentrate on the data
support without necessarily recovering the data distribution.
An explicitly solvable case identifies the limiting weights and their
dependence on Gaussian initialization, while a relative-entropy
estimate quantifies the effect of score approximation.
Together, these findings connect the stability and concentration
mechanisms with score learning and early stopping, providing a
mathematical perspective on the balance between imitation fidelity
and generative diversity.
}
\end{abstract}
\maketitle

\section{Introduction}

The development of generative models, one of the most dynamic areas of contemporary artificial intelligence (AI), aims to endow machines with the ability to create new, realistic samples that are statistically consistent with a given dataset drawn from an unknown distribution.
In the context of image generation, this corresponds to producing novel images that belong to the same underlying class as those in the training set.
More broadly, generative models seek to approximate and sample from complex, high-dimensional data distributions across diverse domains, including text, audio, video, molecular structures, physical fields, and even solutions of partial differential equations (PDEs). 
A central challenge lies in achieving a balance between \emph{imitation fidelity} and \emph{generative diversity}.

Formally, we consider a collection of samples drawn from an unknown data distribution $u_0 \in \mathcal{P}(\R^d)$, where $d$ denotes the data dimensionality and $\mathcal{P}(\R^d)$ is the set of probability measures. 
For instance, $d = 256 \times 256$ for grayscale images, $d = 256 \times 256 \times 3$ for RGB images, and the dimension may be much higher for multimodal data or discretized physical fields in scientific simulations. 
The learning objective is to construct a generator that produces new samples statistically consistent with $u_0$. 
Classical nonparametric density estimation methods, such as \emph{kernel density estimation}~\cite{silverman2018density}, face severe challenges in high-dimensional settings due to the \emph{curse of dimensionality}~\cite[Ch.~4]{wand1994kernel}, which limits their use for modern data domains.

To overcome these limitations, deep generative models have been developed over the past decade. 
Among them, \emph{diffusion models}~\cite{ho2020denoising,song2021scorebased} have emerged as a particularly powerful class. 
These models seek to sample from $u_0$ by learning the \emph{score function} of progressively noised data and modeling the reverse-time dynamics of an underlying diffusion process~\cite{anderson1982reverse}.
The analytical objective of this work is to rigorously characterize, within a PDE framework, the \emph{well-posedness}, \emph{stability}, and \emph{concentration} properties of these generative dynamics.

The forward noising process is governed by a \emph{Fokker--Planck} (FP)
or heat equation. Generation starts from a simple terminal law, usually
Gaussian, and evolves backward using the score as a prescribed drift
in a \textit{stochastic differential equation} (SDE), or in the
corresponding deterministic dynamics.

The present work concerns the stability and small-time behavior of this
reverse evolution. We distinguish concentration on the data support from
recovery of the probability weights on that support. For empirical data,
these questions concern memorization of the training locations and the
frequencies with which they are generated.

\subsection{Formulation of diffusion models}
Diffusion models combine a forward noising process with a backward
generative process. We describe these two processes from the PDE
perspective adopted in this work.

\medskip
\noindent\textit{Forward noising process.}
Given a time horizon $0<T<\infty$ and the ambient space $\R^d$, we first  introduce the \emph{forward heat equation}:
\begin{equation}\label{eq:heat}
\begin{cases}
\partial_t u(x,t) - \Delta u(x,t) = 0,
& (x,t) \in Q,\\[1ex]
u(\cdot,0) = u_0 \in \mathcal{P}(\R^d).
\end{cases}
\end{equation}
Here, $Q$ denotes the space-time cylinder
\[
    Q=\R^d\times(0,T],
    \qquad
    \mathring Q=\R^d\times(0,T),
\]
and $\mathcal P(\R^d)$ is the set of probability measures on $\R^d$.
The initial datum $u_0$ represents the unknown data distribution
and need not admit a Lebesgue density.
We use $\langle\cdot,\cdot\rangle$ and $|\cdot|$ for the Euclidean scalar product and norm, respectively, and
$\operatorname{dist}(x,A):=\inf_{y\in A}|x-y|$ for the distance
from a point to a set.

The solution of \eqref{eq:heat} is given by convolution with
the heat kernel:
\begin{equation}\label{eq:heat_solution}
    u(x,t)=(G_t*u_0)(x),
    \qquad (x,t)\in Q,
\end{equation}
where
\[
    G_t(x)=(4\pi t)^{-d/2}
    \exp\!\left(-\frac{|x|^2}{4t}\right),
    \qquad x\in\R^d,\quad t>0.
\]
This convolution progressively smooths the data distribution.
Since the heat kernel is smooth and strictly positive,
$u$ is smooth and strictly positive at every positive time,
even when $u_0$ is singular. We may therefore introduce
the associated \emph{score function}
\begin{equation}\label{eq:score}
    s(x,t)=\nabla\log u(x,t)
    =\frac{\nabla u(x,t)}{u(x,t)},
    \qquad (x,t)\in Q.
\end{equation}

The score provides the connection between forward noising and
backward generation. Indeed, since $su=\nabla u$, the same
heat-flow solution satisfies, for every $\epsilon\geq0$,
\begin{equation}\label{scorePDE}
    \partial_t u+\epsilon\Delta u
    -(1+\epsilon)\operatorname{div}(su)=0.
\end{equation}
Once the reference heat flow $u$ is fixed, its score $s$ is
regarded as a prescribed vector field. With this convention,
\eqref{scorePDE} is parabolic in the backward time direction
when $\epsilon>0$ and reduces to a transport equation when
$\epsilon=0$. This rewriting motivates a score-driven backward
evolution, distinct from directly inverting the heat equation:
the information about the data distribution is encoded in
the prescribed score.

\medskip
\noindent\textit{Backward generative process.}
Motivated by \eqref{scorePDE}, we define the generative dynamics
by the backward SDE
\begin{equation}\label{intro_eq:generation}
\begin{cases}
dX_t=-(1+\epsilon)s(X_t,t)\,dt
     +\sqrt{2\epsilon}\,dW_t,
& t\in(0,T],\\[1ex]
X_T\sim v_T,
\end{cases}
\end{equation}
where $v_T$ is a probability distribution from which samples
can be drawn easily, such as a Gaussian.
The equation is integrated from $T$ towards zero.
Here, $W_t$ denotes a Brownian motion in backward time,
independent of $X_T$.
When $\epsilon=0$, \eqref{intro_eq:generation} reduces to the
deterministic ODE $\dot X_t=-s(X_t,t)$; when $\epsilon>0$,
it describes stochastic generation.
Unless otherwise stated, our theoretical analysis uses
the exact score defined by \eqref{eq:score}; its approximation
by neural networks is discussed in Section~\ref{sec:discussion}.

In practice, the backward dynamics are discretized using a
numerical method, such as the Euler--Maruyama scheme for
stochastic generation or an ODE solver in the deterministic case.
The idealized output distribution is the limiting law as
$t\to0^+$, whose existence and relation to $u_0$ are central
questions of our analysis.

At the level of probability laws, the associated backward
FP equation for $v(t)=\operatorname{law}(X_t)$ takes the form
\begin{equation}\label{eq:FP}
\begin{cases}
\partial_t v(x,t)+\epsilon\Delta v(x,t)
-(1+\epsilon)\operatorname{div}\bigl(s(x,t)v(x,t)\bigr)=0,
& (x,t)\in\mathring Q,\\[1ex]
v(\cdot,T)=v_T.
\end{cases}
\end{equation}

Two features make the small-time behavior of this evolution
nontrivial: the score may become singular as $t\to0^+$, and the
terminal initialization $v_T$ need not coincide with the forward
law $u(T)$.

\subsection{Main results}\label{sec:intro-main}

For a compactly supported data measure $u_0$, possibly singular,
we study the stability of \eqref{eq:FP} and the concentration of
the generated law near the singular endpoint $t=0$. We establish
three main results.

\begin{itemize}
\setlength{\itemsep}{0.7em}
\item \textbf{Stability of the generated densities
(Theorem~\ref{thm:energy_score}).}
The Li--Yau inequality \cite{li1986parabolic} provides the one-sided bound
\[ 
    \operatorname{div}s(x,t)\geq-\frac{d}{2t}
\]
on the divergence of the score function. Combining this estimate
with an $L^p$ energy argument for the backward FP equation, we obtain,
for $1\leq p<\infty$, every $\epsilon\geq0$, and terminal
probability densities $v_T\in L^p(\R^d)$,
\[
    \|v(t)\|_{L^p}
    \leq
    \left(\frac{T}{t}\right)^{\frac{d(1+\epsilon)(p-1)}{2p}}
    \|v_T\|_{L^p},
    \qquad 0<t\leq T.
\]
The amplification factor is sharp in general as a uniform bound over
terminal data; see Remark~\ref{rem:sharpness}.

\item \textbf{Convergence of the generated laws
(Theorem~\ref{thm:support}).}
We prove that, as $t\to0^+$, the generated law converges to
a limiting distribution $v_0^*$ supported on the data support.
This holds for both deterministic and stochastic generation,
provided that the terminal law $v_T$ has a positive density,
finite absolute entropy, and a finite second moment.
More precisely, the generated laws converge in the
$2$-Wasserstein distance with the rate
\[
    W_2\bigl(v(t),v_0^*\bigr)\leq C_T\sqrt t,
    \qquad 0<t\leq T.
\]
Under the additional condition on $v_T$ stated in the theorem,
which is satisfied by every nondegenerate Gaussian initialization,
the limit has exactly the same support as the data:
\[
    \operatorname{supp}(v_0^*)=\operatorname{supp}(u_0).
\]
The limiting law may nevertheless differ from $u_0$, since its weights
depend on the terminal initialization.

\item \textbf{Convergence of individual trajectories
(Theorem~\ref{thm:convergence_rate}).}
For deterministic generation, we prove that for Lebesgue-almost every
terminal point, the trajectory converges to a point of the data support
as $t\to0^+$, for an arbitrary compactly supported data measure.
Moreover, in the empirical case
\( 
    \operatorname{supp}(u_0)=\{y_i\}_{i=1}^N,
\)
the limiting point is one of the training samples and the convergence
is quantitative:
\[
    |X_t-y_i|\leq C\sqrt t
\]
for some $i\in\{1,\ldots,N\}$.
\end{itemize}

Two complementary results identify the limiting weights in the explicit
case $\epsilon=1$ (Section~\ref{sec:explicit}) and quantify the error
caused by score approximation (Proposition~\ref{prop:score-stability}).

\subsection{Related work and our contribution}\label{sec:compare}

We discuss the relation between our results and existing analyses
of density stability, endpoint convergence, and memorization.

\medskip
\noindent\textit{Stability of the generative equation.}
Positive-time well-posedness follows from the classical theory of
FP and transport equations
\cite{lebris2007fokker,bogachev2022fokker}.
Theorem~\ref{thm:energy_score} provides quantitative control as
the score becomes singular near $t=0$: the one-sided Li--Yau
bound on its divergence yields an explicit $L^p$ amplification
factor that is sharp in general as a uniform bound over terminal data.
Positivity of the evolution then gives the corresponding
stability estimate for differences of probability solutions
in Remark~\ref{rem:Lp-stability}.

For approximate scores, \cite[Thm.~1]{song2021maximum} bounds
the data-to-generated Kullback--Leibler (KL) divergence by the terminal mismatch
and a likelihood-weighted score-matching loss.
In the heat-flow normalization, its coefficient agrees with
Proposition~\ref{prop:score-stability} when $\epsilon=1$.
Our proposition places this established connection within the
present PDE framework, treating every $\epsilon>0$ through
a direct entropy argument on positive-time intervals.

\medskip
\noindent\textit{Convergence of the generated laws.}
For exact reverse diffusions,
\cite[Thm.~1]{pidstrigach2022manifolds} establishes endpoint
existence under its time-reversal hypotheses and shows that absolute continuity of the terminal
initialization with respect to the forward terminal law is
inherited at the data endpoint. Equivalent terminal laws give
equivalent endpoint laws, and the corresponding $f$-divergences
contract. Thus support recovery under a change of initialization
is already established in this stochastic setting.

Theorem~\ref{thm:support} complements these endpoint results
with the quantitative estimate
$W_2(v(t),v_0^*)\leq C_T\sqrt t$.
A common entropy and continuity-equation argument covers
every $\epsilon\geq0$, including deterministic generation,
under the stated terminal-density, entropy, and moment
assumptions. The data measure may have an arbitrary compact
support, without manifold or other geometric regularity.

Gaussian approximation of the forward terminal law by moment
matching is considered in
\cite[Lem.~2]{pidstrigach2022manifolds}.
Our explicit analysis at $\epsilon=1$ instead quantifies its
effect on the generated endpoint. The expansion of the weight
$a_{v_T}$ in $dv_0^*=a_{v_T}\,du_0$, through order $T^{-2}$,
identifies how Gaussian initialization changes the sampling
weights and why matching its center to the data barycenter
cancels the first-order correction.

\medskip
\noindent\textit{Ornstein--Uhlenbeck noising and sampling estimates.}
Several sampling analyses use Ornstein--Uhlenbeck (OU) noising
rather than the heat flow
\cite{lee2022convergence,chen2023sampling,benton2024nearly,
conforti2025kl}.
After a normalization of time, the forward equation is
\begin{equation}\label{eq:OU}
    \partial_\tau w-\Delta w-\operatorname{div}(xw)=0,
    \qquad w(\cdot,0)=u_0,
\end{equation}
whose stationary probability density is $G_{1/2}$.
The change to self-similar variables
\cite{barenblatt1996scaling,zuazua2020asymptotic}
gives the relation
\[
    u(x,t)
    =(1+2t)^{-d/2}
    w\!\left(
        \frac{x}{\sqrt{1+2t}},
        \frac{\log(1+2t)}2
    \right).
\]
Thus heat and OU noising are related by a deterministic change
of time and space, and the singular small-noise endpoint occurs
in both formulations.

The cited sampling results control initialization, score
approximation, and discretization errors in total variation
or relative entropy under their respective assumptions.
For compactly supported, possibly singular targets,
\cite[Thm.~1]{debortoli2022manifold} gives a $W_1$ sampling
bound under weighted score-error and step-size conditions,
including an early-stopping error.
Our estimates address the continuous exact-score dynamics
at a fixed horizon and quantify its approach to its own
endpoint. They are therefore complementary to numerical
sampling guarantees, rather than bounds on the same error
under weaker assumptions.

\medskip
\noindent\textit{Convergence of individual trajectories.}
For empirical-score dynamics,
\cite[Thm.~4.3]{baptista2025memorization} uses Voronoi geometry
to characterize possible limit points. If a limit point lies
in the interior of a Voronoi cell, the whole trajectory
converges to the corresponding training sample at the scale
of the forward noise standard deviation. The theorem also
allows limit points on Voronoi boundaries. In
\cite[Sect.~4.2]{baptista2025memorization}, the authors conjecture
that Gaussian initialization almost surely excludes these
boundary limit points.

Theorem~\ref{thm:convergence_rate} establishes this conjectured
conclusion in the heat-flow setting. For empirical data,
Lebesgue-almost every terminal point gives a trajectory
converging to a training sample, with the
$\mathcal O(\sqrt t)$ rate corresponding to the noise scale.
The conclusion therefore holds almost surely for any
absolutely continuous terminal law, including every
nondegenerate Gaussian.
More generally, part~\emph{(i)} proves convergence to a point
of the data support for an arbitrary compactly supported
measure. The finite-mean-length argument in
Section~\ref{sec:proof-trajectories} requires neither manifold
regularity nor a lower-mass condition; Voronoi geometry is
used only to obtain the empirical rate.

A complementary direction is pursued in
\cite[Thm.~2.5]{lyu2025resolving}, where an inertia update
replaces the final stage of empirical diffusion to mitigate
memorization. Under a compact $C^2$ manifold assumption and
regularity and positivity conditions on the data density,
the modified method admits a high-probability $W_1$ approximation bound
with a sample-size exponent governed by the intrinsic dimension,
for the terminal initializations and horizon choices specified there.
Our theorem concerns the unmodified dynamics and describes
the endpoint concentration that motivates such modifications.
Further aspects of empirical-score accuracy, large-dimensional
dynamics, and training-data copying are studied
in~\cite{li2024good,biroli2024regimes,
carlini2023extracting,somepalli2023understanding}.

\subsection{Organization}
Section~\ref{sec:prelimiaires} establishes positive-time well-posedness
and $L^p$ stability of the generative Fokker--Planck equation.
Section~\ref{sectionSC} studies the small-time asymptotics of the
probability laws and deterministic trajectories, followed by the
explicit case $\epsilon=1$.
Section~\ref{sec:discussion} relates score approximation to training
and regularization. Section~\ref{sec:proof} begins with the heat-flow
estimates and the classical well-posedness argument, and then presents
the energy, endpoint, trajectory, and entropy calculations.
Section~\ref{sec:conclusion} concludes the paper.
Appendices~\ref{app:wellposed} and~\ref{app:entropy-justification}
provide the localization and approximation arguments for the stability
estimates; Appendix~\ref{app:explicit} contains the explicit-case proofs.

\section{Well-posedness and stability of the generative Fokker--Planck equation}
\label{sec:prelimiaires}

We study the backward FP equation~\eqref{eq:FP} with the prescribed
score
\[
    s=\nabla\log u,
    \qquad
    u(t)=G_t*u_0,
\]
associated with the reference heat flow.
We first establish well-posedness on intervals bounded away from $t=0$.
We then use the Li--Yau inequality to derive $L^p$ estimates and
stability with respect to the terminal law.

\subsection{Existence and uniqueness}\label{sec:wellposed}

Assume that $u_0\in\mathcal P(\R^d)$ has compact support contained in
$B_R(0)$ for some $R>0$. By Lemma~\ref{lem:gaussian-bounds}, the score
is smooth at positive times and satisfies
\[
    |s(x,t)|\leq\frac{|x|+R}{2t},
    \qquad (x,t)\in Q.
\]
Moreover, the Hessian estimate in Lemma~\ref{lem:gaussian-bounds}
gives, for every $0<\tau<T$,
\[
    \sup_{t\in[\tau,T]}
    \|\nabla_x s(\cdot,t)\|_{L^\infty(\R^d)}<\infty.
\]
Hence, on each interval $[\tau,T]$ bounded away from zero, the drift
is globally Lipschitz in space and has at most linear growth. These
properties place \eqref{eq:FP} within the classical positive-time
well-posedness theory.

We equip $\mathcal P(\R^d)$ with the topology of narrow convergence,
that is, convergence against bounded continuous functions. For probability
measures, $\rightharpoonup$ denotes convergence in this topology. We denote
by $\mathcal P_2(\R^d)$ the space of probability measures with finite
second moment, endowed with the $2$-Wasserstein distance $W_2$;
see~\cite[Ch.~6]{villani2008optimal}.

A narrowly continuous curve
$v:(0,T]\to\mathcal P(\R^d)$ is called a weak solution of
\eqref{eq:FP} with terminal law $v_T$ if $v(T)=v_T$ and
\begin{equation}\label{eq:weak-FP}
    \int_0^T\!\int_{\R^d}
    \left(
        -\partial_t\varphi
        +\epsilon\Delta\varphi
        +(1+\epsilon)\langle s,\nabla\varphi\rangle
    \right)\,dv(t)\,dt=0,
    \qquad
    \varphi\in C_c^\infty(\mathring Q).
\end{equation}

\begin{prop}[Well-posedness]
\label{prop:wellposed}
Let $T>0$, $\epsilon\geq0$, and let
$u_0\in\mathcal P(\R^d)$ have compact support.
\begin{enumerate}
\item[\emph{(i)}]
For every $v_T\in\mathcal P(\R^d)$, there exists a unique weak
solution of \eqref{eq:FP} in the sense of \eqref{eq:weak-FP} such that
\[
    v\in C\bigl((0,T];\mathcal P(\R^d)\bigr),
    \qquad
    v(T)=v_T.
\]
If $\epsilon>0$, the solution admits a smooth, strictly positive
density on $\mathring Q$.

\item[\emph{(ii)}]
If, in addition, $v_T\in\mathcal P_2(\R^d)$, then
\[
    v\in C\bigl((0,T];\mathcal P_2(\R^d),W_2\bigr).
\]
\end{enumerate}
\end{prop}

\begin{proof}
The proof is given in Section~\ref{sec:proof-wellposed}.
\end{proof}

For $\epsilon>0$, the classical correspondence between diffusion
processes and FP equations identifies this weak solution with the
time marginals of the generative SDE~\eqref{intro_eq:generation};
see~\cite[Sect.~1.3]{bogachev2022fokker}.

\subsection{The Li--Yau inequality and \texorpdfstring{$L^p$}{Lp} stability}
\label{sec:preliminary-stability}

The preceding growth and regularity estimates ensure well-posedness
on every interval bounded away from zero. To control the evolution
as $t\to0^+$, where the score may become singular, we use a
one-sided estimate on its divergence.

The classical Li--Yau inequality~\cite{li1986parabolic} for positive
solutions of the heat equation gives
\begin{equation}\label{intro_eq:Li-Yau-1}
    \Delta\log u(x,t)\geq-\frac{d}{2t},
    \qquad (x,t)\in Q.
\end{equation}
Since $s=\nabla\log u$, this is equivalent to
\begin{equation}\label{intro_eq:Li-Yau-2}
    \operatorname{div}s(x,t)\geq-\frac{d}{2t},
    \qquad (x,t)\in Q.
\end{equation}
In the present Euclidean setting, the same estimate also follows
directly from the Gaussian covariance identity established in the
proof of Lemma~\ref{lem:gaussian-bounds}.

In contrast with the spatial Lipschitz estimate used for
well-posedness, \eqref{intro_eq:Li-Yau-2} controls only the negative
part of the divergence. This is precisely the quantity entering the
$L^p$ energy argument and leads to the following estimate.

\begin{thm}[Energy estimate for the score-based FP equation]
\label{thm:energy_score}

Let $u_0\in\mathcal P(\R^d)$ have compact support, let $T>0$,
$\epsilon\geq0$, and $1\leq p<\infty$. Suppose that
$v_T\in\mathcal P(\R^d)$ admits a density in $L^p(\R^d)$.
Then the unique weak solution of
Proposition~\ref{prop:wellposed}\emph{(i)} admits a density, still
denoted by $v$, such that
\[
    v\in C\bigl((0,T];L^p(\R^d)\bigr),
    \qquad
    v(T)=v_T,
\]
and

\begin{equation}\label{eq:energy_score}
    \|v(t)\|_{L^p}
    \leq
    \left(\frac{T}{t}\right)^{
        \frac{d(1+\epsilon)(p-1)}{2p}}
    \|v_T\|_{L^p},
    \qquad 0<t\leq T.
\end{equation}
\end{thm}

\begin{proof}
The key argument, showing how the Li--Yau inequality enters the
$L^p$ energy estimate, is presented in
Section~\ref{sec:proof-Lp} under auxiliary regularity assumptions.
The rigorous proof under the hypotheses above, based on spatial
cutoffs and approximation, is given in
Appendix~\ref{app:wellposed}.
\end{proof}

\begin{rem}[Stability with respect to the terminal law]
\label{rem:wellposedness}\label{rem:Lp-stability}
Under the assumptions on $u_0,T,\epsilon$, and $p$ in
Theorem~\ref{thm:energy_score}, let $v_{1,T}$ and $v_{2,T}$
be probability densities in
$L^p(\R^d)$. The corresponding weak solutions of \eqref{eq:FP},
with the same score field, satisfy
\begin{equation}\label{eq:Lp-stability}
    \|v_1(t)-v_2(t)\|_{L^p}
    \leq
    \left(\frac{T}{t}\right)^{
        \frac{d(1+\epsilon)(p-1)}{2p}}
    \|v_{1,T}-v_{2,T}\|_{L^p},
    \qquad 0<t\leq T.
\end{equation}
In particular, the evolution is contractive in $L^1$.
The estimate follows from positivity and linearity of the solution
operator together with Theorem~\ref{thm:energy_score}; see
Section~\ref{sec:proof-Lp}.
\end{rem}

\begin{rem}[Sharpness of the estimate]\label{rem:sharpness}
The factor in \eqref{eq:energy_score} is optimal in general as a
uniform bound over terminal probability densities, for every
$\epsilon\geq0$.
Indeed, consider the Gaussian example
\begin{align*}
    u_0&=\delta_0,
    & s(x,t)&=-\frac{x}{2t},
    & v_T^M&=G_{M/2},\qquad M>0.
\end{align*}
The corresponding solution is
\[
    v^M(t)=G_{B_M(t)/2},
    \qquad
    B_M(t)
    =M\left(\frac{t}{T}\right)^{1+\epsilon}
     +2t\left(1-\left(\frac{t}{T}\right)^\epsilon\right),
\]
where $B_M$ solves
\[
    B_M'(t)=\frac{1+\epsilon}{t}B_M(t)-2\epsilon,
    \qquad B_M(T)=M.
\]
Setting $\alpha_p=d(p-1)/(2p)$, for every fixed $0<t<T$ we obtain
\[
    \lim_{M\to\infty}
    \frac{\|v^M(t)\|_{L^p}}{\|v_T^M\|_{L^p}}
    =
    \lim_{M\to\infty}
    \left(\frac{M}{B_M(t)}\right)^{\alpha_p}
    =
    \left(\frac{T}{t}\right)^{(1+\epsilon)\alpha_p},
\]
which is exactly the factor in \eqref{eq:energy_score}.
Hence this factor cannot be improved uniformly over the terminal data.

This sharpness statement is uniform in the terminal datum and does not
mean that every fixed datum has the same small-time behavior. For
instance, if $M$ is fixed and $\epsilon>0$, then
$B_M(t)\sim2t$ as $t\to0^+$. For $p=1$, the estimate reduces to
conservation of mass.
\end{rem}

\begin{figure}[htbp]
  \centering
  \includegraphics[width=\textwidth]{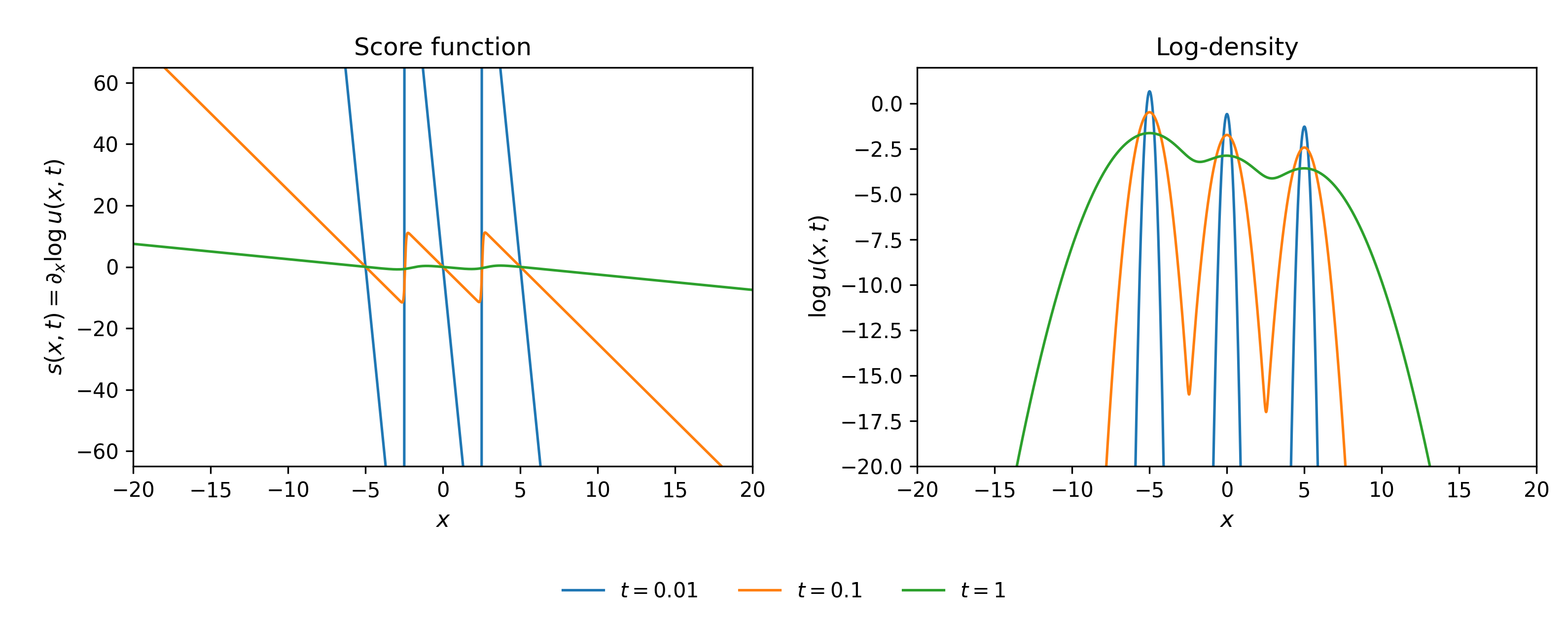}
  \caption{Score function (left) and log-density (right) of the heat flow
  originating from $u_0=0.7\delta_{-5}+0.2\delta_0+0.1\delta_5$,
  shown at $t=0.01,0.1,1$. The steep negative score slopes near the atoms
  illustrate the $1/t$ scale in the Li--Yau estimate. The log-density and
  score are evaluated using log-sum-exp and normalized Gaussian weights,
  respectively.}
  \label{fig:score_density}
\end{figure}

\section{Small-time asymptotics of the generative process}\label{sectionSC}
We now study the behavior of the generative dynamics as $t\to0^+$.
We first establish concentration and convergence of the generated laws,
and then analyze individual deterministic trajectories. The explicit
case $\epsilon=1$ further describes the dependence on terminal
initialization.

\subsection{Endpoint convergence and concentration on the data support}
To analyze the endpoint behavior of the diffusion model, we study
whether its solution concentrates on $\operatorname{supp}(u_0)$ as
$t\to0^+$. We address this question by an entropy stability analysis
of the FP equation.

Recall that the KL divergence (or relative entropy) between two 
probability measures $\mu_1, \mu_2 \in \mathcal{P}(\R^d)$ is defined by
\begin{equation}\label{eq:KL-def}
     \mathrm{KL}(\mu_1 \,\|\, \mu_2) \;=\; 
     \begin{cases}
         \displaystyle \int_{\R^d} \log\!\left(\frac{d\mu_1}{d\mu_2}\right)\,d\mu_1, 
         & \text{if } \mu_1 \ll \mu_2, \\[2ex]
         +\infty, & \text{otherwise},
     \end{cases}
\end{equation}
where $\mu_1 \ll \mu_2$ means that $\mu_1$ is absolutely continuous with respect to $\mu_2$.

As we shall see  in Lemma \ref{lem:KL},  the KL divergence between \(v\) (the solution of~\eqref{eq:FP}) and \(u\) (the reference solution of the heat equation defining the score function) decays when viewed backward in time:
\begin{equation}
   \mathrm{KL}\big(v(t_1)\,\|\, u(t_1)\big) \;\leq\;  \mathrm{KL}\big(v(t_2)\,\|\, u(t_2)\big), 
   \qquad \forall\, 0 <t_1\leq t_2 \leq T.
\end{equation}
The entropy controls the generated-law score estimate used in the
proof of the following theorem. In addition to positive-time
well-posedness, the endpoint argument requires the following density
and entropy assumptions.

\Needspace{9\baselineskip}

\begin{thm}[Concentration of support and endpoint convergence]\label{thm:support}
Let $T>0$, $\epsilon\geq0$, and let $u_0\in\mathcal P(\R^d)$ have
compact support $S$. Assume that $v_T\in\mathcal P_2(\R^d)$ has a
density satisfying
\begin{equation}\label{eq:terminal-entropy-assumption}
 v_T>0\quad\mathcal L^d\text{-almost everywhere},\qquad
 \int_{\R^d}v_T(x)|\log v_T(x)|\,dx<\infty.
\end{equation}
Let $v$ be the solution of \eqref{eq:FP}, and set
\[
 K_T:=\mathrm{KL}(v_T\|u(T)),\qquad
 C_T:=2\sqrt{2K_T+d\log2}+\sqrt{\epsilon K_T}.
\]
Then $K_T<\infty$, and the following assertions hold.
\begin{enumerate}
\item[\emph{(i)}] There exists $v_0^*\in\mathcal P_2(\R^d)$ such that
\begin{equation}\label{eq:W2-endpoint-rate}
 W_2\bigl(v(t),v_0^*\bigr)\leq C_T\sqrt t,
 \qquad 0<t\leq T.
\end{equation}
\item[\emph{(ii)}] One has $v_0^*\ll u_0$ and
$\operatorname{supp}(v_0^*)\subset S$.
If there exist $C,q>0$ such that
$|\log v_T(x)|\leq C(1+|x|^q)$ for $\mathcal L^d$-almost every $x$,
then $u_0\ll v_0^*$ and
\[
 \operatorname{supp}(v_0^*)=\operatorname{supp}(u_0)=S.
\]
\end{enumerate}
\end{thm}

\begin{proof}
The proof is presented in Section~\ref{sec:proof-endpoint}.
\end{proof}

\begin{rem}[Deterministic generation]
For $\epsilon=0$, the generative dynamics follow the gradient of
the time-dependent potential $\log u$ in the backward time direction. Although local maximizers of
$\log u(\cdot,t)$ at positive times need not belong to
$\operatorname{supp}(u_0)$, Theorem~\ref{thm:support} shows that the
generated law concentrates on the data support as $t\to0^+$.
At the trajectory level, Theorem~\ref{thm:convergence_rate} further
shows that Lebesgue-almost every deterministic trajectory converges to
a point of $\operatorname{supp}(u_0)$; for empirical data, the
convergence occurs toward a training sample with rate
$\mathcal O(\sqrt t)$.
\end{rem}

\begin{rem}[Concentration and dependence on initialization]
\label{rem:concentration}\label{rem:endpoint-TV}
Under the assumptions of Theorem~\ref{thm:support}, since
$\operatorname{supp}(v_0^*)\subset S$, we have
\begin{equation}\label{eq:mean-square-support-rate}
 \left(\int_{\R^d}\operatorname{dist}(x,S)^2\,dv(t)(x)\right)^{1/2}
 \leq W_2\bigl(v(t),v_0^*\bigr)
 \leq C_T\sqrt t.
\end{equation}
Consequently, for every $r>0$,
\begin{equation}\label{eq:concentration-tail}
 v(t)\{x:\operatorname{dist}(x,S)\geq r\}
 \leq \min\left\{1,\frac{C_T^2t}{r^2}\right\}.
\end{equation}
In particular, $v(t)(U)\to1$ for every open neighborhood
$U$ of $S$.

The $L^1$ stability estimate also controls the dependence of the
endpoint on the terminal initialization. Indeed, since
$v(t)\to v_0^*$ and $u(t)\to u_0$ narrowly as $t\to0^+$,
lower semicontinuity of the total variation norm and
Remark~\ref{rem:Lp-stability} with $p=1$ give
\begin{equation}\label{eq:endpoint-TV}
 \|v_0^*-u_0\|_{\mathrm{TV}}
 \leq
 \|v_T-u(T)\|_{L^1}.
\end{equation}
Here $\|\sigma\|_{\mathrm{TV}}:=|\sigma|(\R^d)$, so that
$\|f\,dx\|_{\mathrm{TV}}=\|f\|_{L^1}$.
More generally, the same argument compares the limiting laws
associated with any two admissible terminal densities.
\end{rem}

\begin{rem}[Numerical experiment in Figure~\ref{fig:diffusion_model}]
A visualization of this concentration phenomenon is presented in Figure~\ref{fig:diffusion_model}, which illustrates the dynamics in two spatial dimensions.
We consider a true data distribution supported on a one-dimensional curve, represented by the lemniscate.

Panels (A)-(B) show trajectories of the generation flow~\eqref{intro_eq:generation} driven by the exact score drift corresponding to this full distribution. In panel (A), we observe how the trajectories of the generative SDE return to the support of $u_0$. In panel (B), we display the locations of the trajectories when the dynamics are stopped prematurely, before reaching $t = 0$. The displayed positive-time samples lie away from the support; no monotonicity in the stopping time is asserted.

Panels (C)-(D) depict analogous trajectories when the generation flow~\eqref{intro_eq:generation} is driven by the empirical score drift, obtained from an approximation of the full distribution $u_0$ by the average of several Dirac realizations. In panel (C), the trajectories of the generative SDE return to the support of this empirical approximation, while panel (D) shows the state of the trajectories under early stopping. This illustrates the discrepancy caused by stopping before the singular endpoint.

The two experiments illustrate concentration toward different reference
supports: a curve for the population measure and finitely many points
for the empirical measure. This comparison does not by itself quantify
statistical generalization.
\end{rem}

\begin{figure}[htbp]
  \centering

  \begin{subfigure}[t]{0.49\textwidth}
    \centering
    \includegraphics[width=\linewidth]{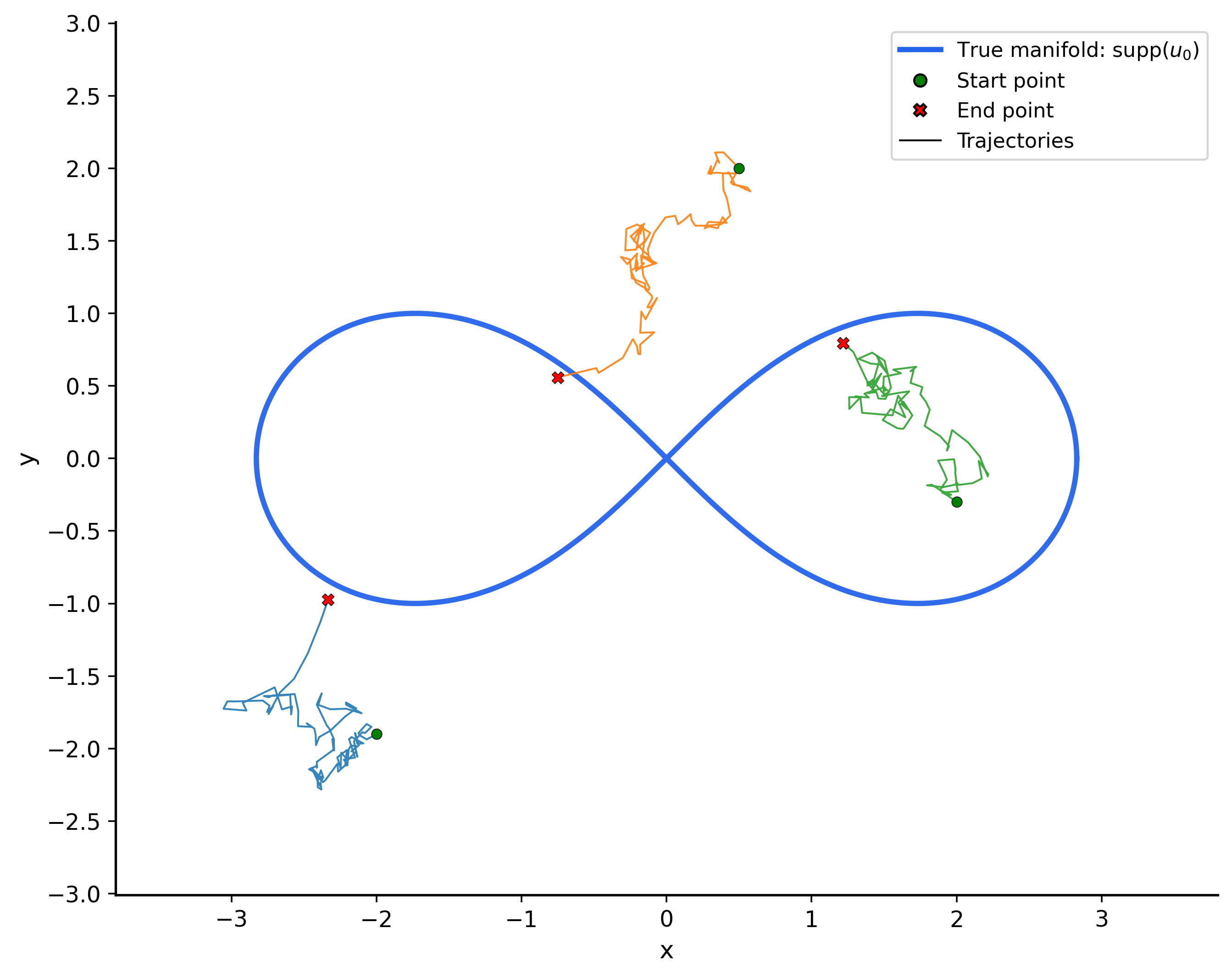}
    \caption{Three backward stochastic trajectories with the true score function.}
  \end{subfigure}
  \hfill
  \begin{subfigure}[t]{0.49\textwidth}
    \centering
    \includegraphics[width=\linewidth]{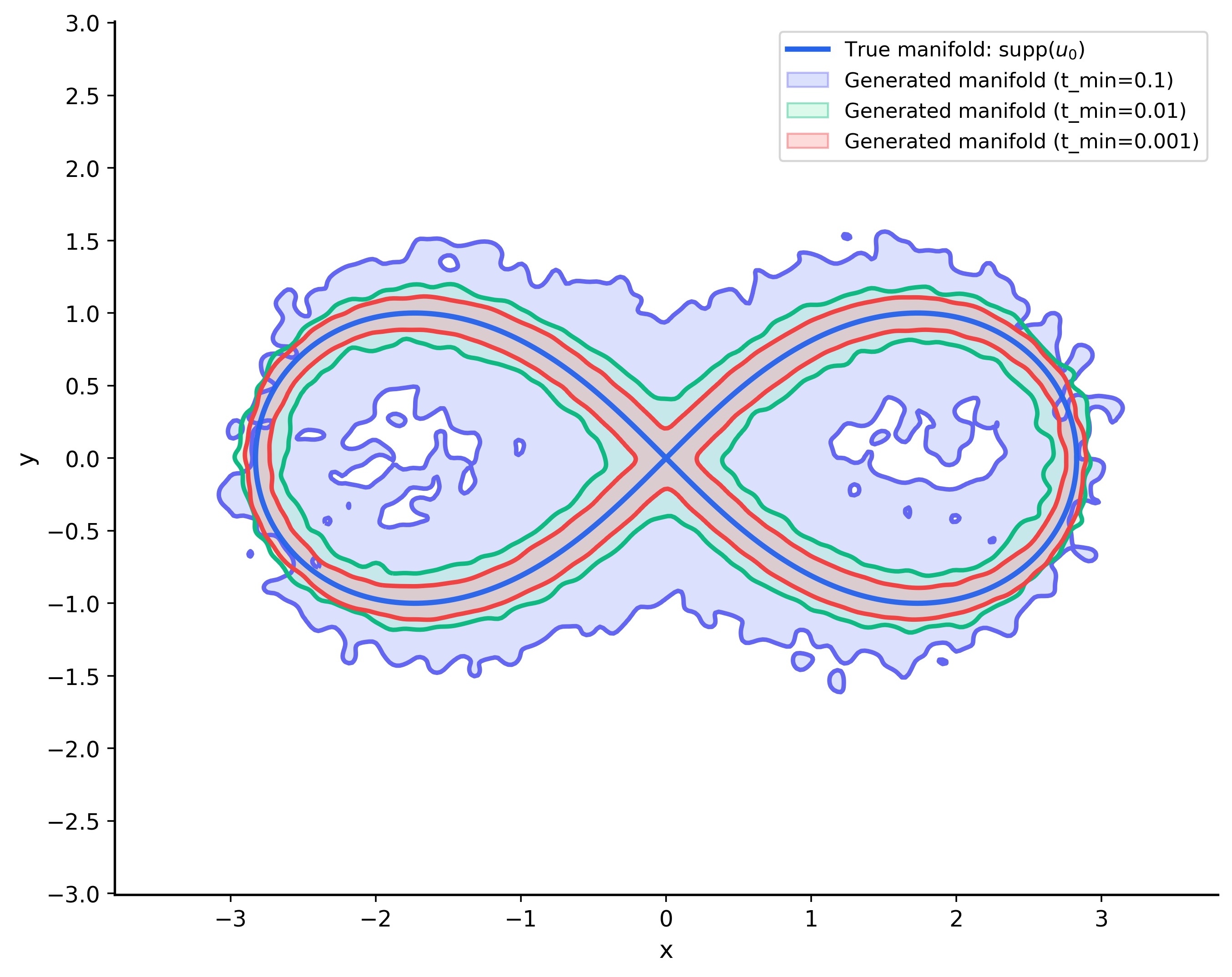}
    \caption{Early-stopped samples with the true score function.}
  \end{subfigure}

  \vspace{0.6em}

  \begin{subfigure}[t]{0.49\textwidth}
    \centering
    \includegraphics[width=\linewidth]{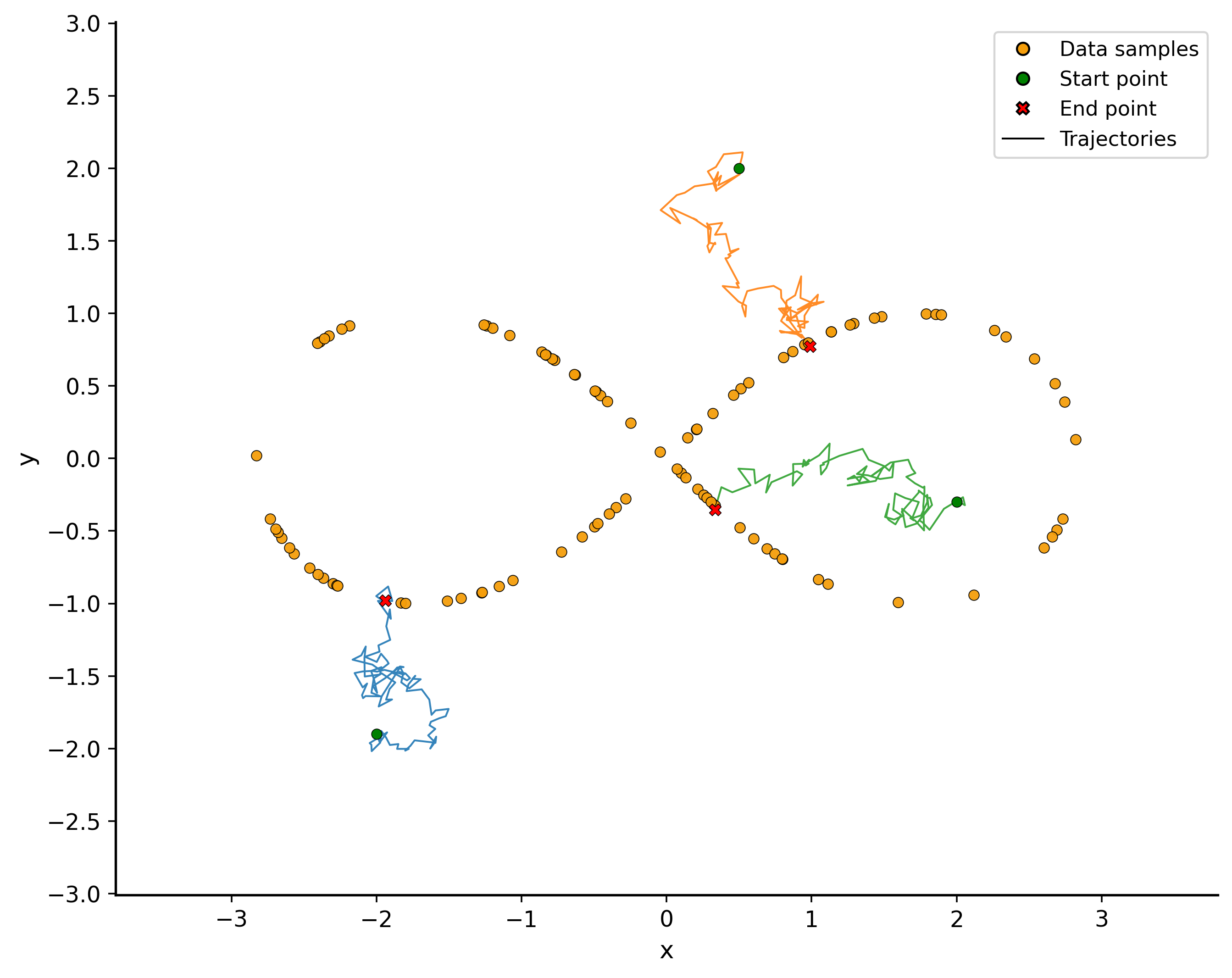}
    \caption{Three backward stochastic trajectories with the empirical score function of data samples.}
  \end{subfigure}
  \hfill
  \begin{subfigure}[t]{0.49\textwidth}
    \centering
    \includegraphics[width=\linewidth]{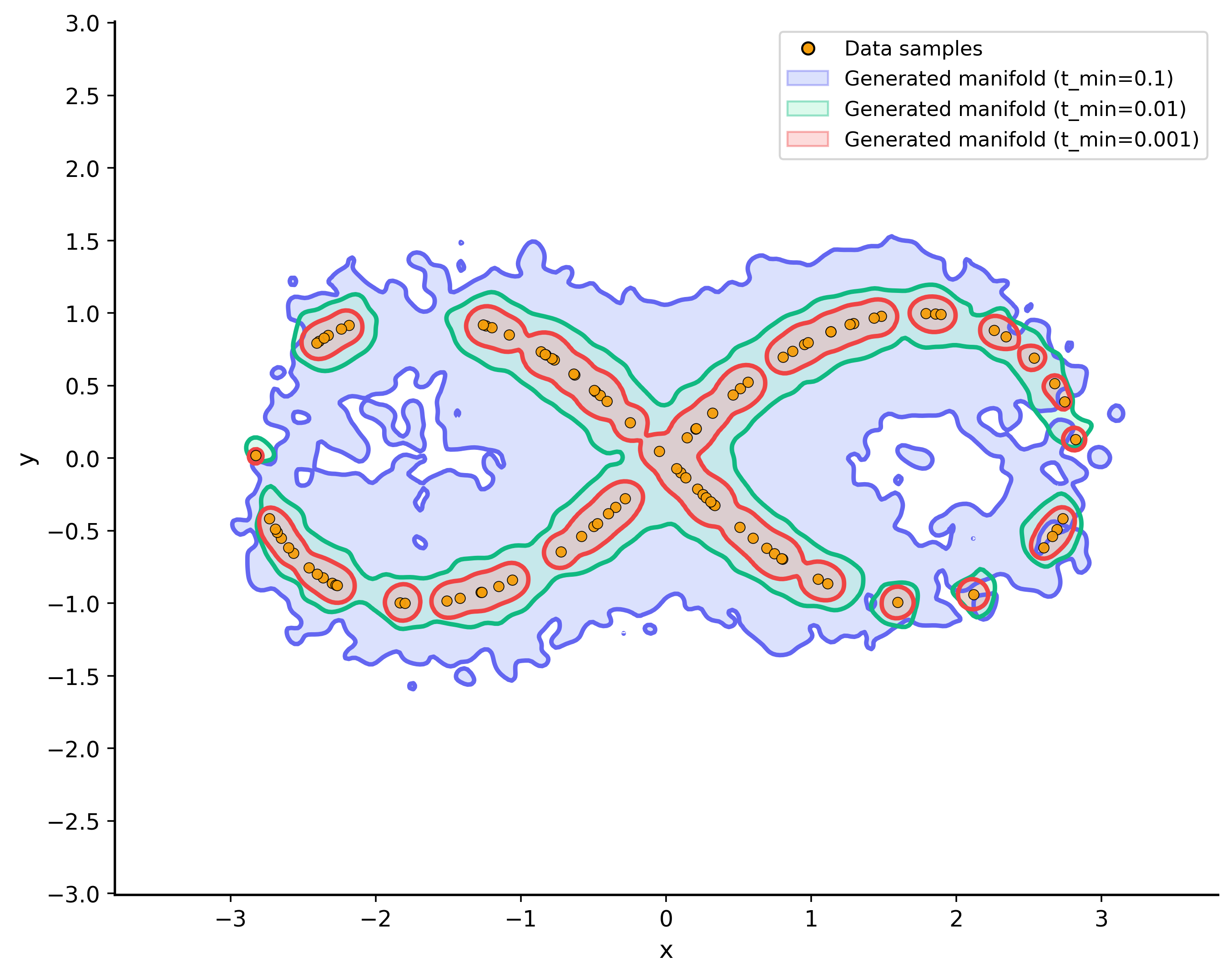}
    \caption{Early-stopped samples with the empirical score function.}
  \end{subfigure}

  \caption{Score-based generation on the lemniscate dataset: comparison between the true ((A)-(B)) and empirical scores ((C)-(D)), and the effect of the choice of the stopping time \(t_{\min} \in \{0.1,\,0.01,\,0.001\}\). 
The true score  ((A)-(B)) corresponds to the solution of the heat equation, given by the convolution with the Gaussian heat kernel, computed by means of numerical quadrature, while the empirical score ((C)-(D)) is obtained from a finite sample of data points via the explicit expression~\eqref{eq:score_empirical}. 
In all experiments, we set the backward diffusion parameter to be \(\epsilon = 0.2\).  The displayed sample region encompasses \(95\%\) of the \(10{,}000\) reverse-time samples, reducing the effect of extreme stochastic trajectories.
}
\label{fig:diffusion_model}

\end{figure}

\begin{rem}[Empirical score function]
If the initial distribution is an empirical measure supported on the finite set $\{y_1,\dots,y_N\}$,
\begin{equation*}
u_0 \;=\; \frac{1}{N}\sum_{k=1}^N\,\delta_{y_k},
\end{equation*}
then the exact (empirical) score admits the explicit formula
\begin{equation}\label{eq:score_empirical}
s(x,t)
=\frac{1}{2t}\!\left(
\frac{\sum_{k=1}^N e^{-\frac{|x-y_k|^2}{4t}}\,(y_k-x)}
     {\sum_{k=1}^N e^{-\frac{|x-y_k|^2}{4t}}}
\right)=\frac{1}{2t}\!\left(
\frac{\sum_{k=1}^N e^{-\frac{|x-y_k|^2}{4t}}\,y_k}
     {\sum_{k=1}^N e^{-\frac{|x-y_k|^2}{4t}}}
- x\right),
\end{equation}
for all $(x,t)\in Q.$

Theorem~\ref{thm:convergence_rate} shows that almost every deterministic
trajectory converges to one of these training points. This does not assert
that their sampling weights are recovered.
\end{rem}

\begin{rem}[Examples of terminal-law dependence]\label{rem:terminal-examples}

Nondegenerate Gaussian densities satisfy the hypotheses of Theorem~\ref{thm:support}.
The entropy and moment conditions ensure
$\mathrm{KL}(v_T\|u(T))<\infty$; they are sufficient conditions, not
claimed to be necessary. Strict positivity is not needed for support
inclusion, but is relevant to recovering the entire support. The
additional global bound $|\log v_T(x)|\leq C(1+|x|^q)$ is a convenient
sufficient condition for finite reverse entropy. The following examples
illustrate possible failures for singular initializations or for laws
confined to one basin.

Let
\[
    u_0 = \tfrac{1}{2}\bigl( \delta_{(-1,0)} + \delta_{(1,0)} \bigr),
\]
and consider the deterministic flow case (\(\epsilon = 0\)):

\begin{enumerate}
    \item \textit{Case 1.}  
    Suppose that \(\operatorname{supp}(v_T) \subset \{(0,y) \,:\, y \in \mathbb{R}\}\).  
    Then \(v_T\) does not possess a density in \(\mathbb{R}^2\).  
    From \eqref{eq:score_empirical}, one verifies that
\[
        -s((0,y),t) = (0, y/(2t)), \qquad \forall\, y \in \mathbb{R}.
    \]

    Consequently, the generative flow collapses every point in \(\operatorname{supp}(v_T)\) to the single point \((0,0)\),  
    which lies outside \(\operatorname{supp}(u_0)\),
    see the vertical-axis trajectories in Figure~\ref{fig:convergence}.

    \item \textit{Case 2.}  
   Assume that $v_T\ll\mathcal L^2$ and
    $v_T((0,\infty)\times\R)=1$, so the terminal density is not
    strictly positive on $\R^2$.
    By symmetry and uniqueness, the right half-plane is invariant
    under the deterministic flow. The almost-everywhere convergence
    in Theorem~\ref{thm:convergence_rate}\emph{(i)} therefore implies
    that $v_T$-almost every trajectory converges to $(1,0)$.
    Consequently, the limiting law is $\delta_{(1,0)}$, rather than
    a law supported on both data points; see the right-hand
    trajectories in Figure~\ref{fig:convergence}.
\end{enumerate}
\end{rem}

\begin{figure}[htbp]
  \centering
  \includegraphics[width=0.5\textwidth]{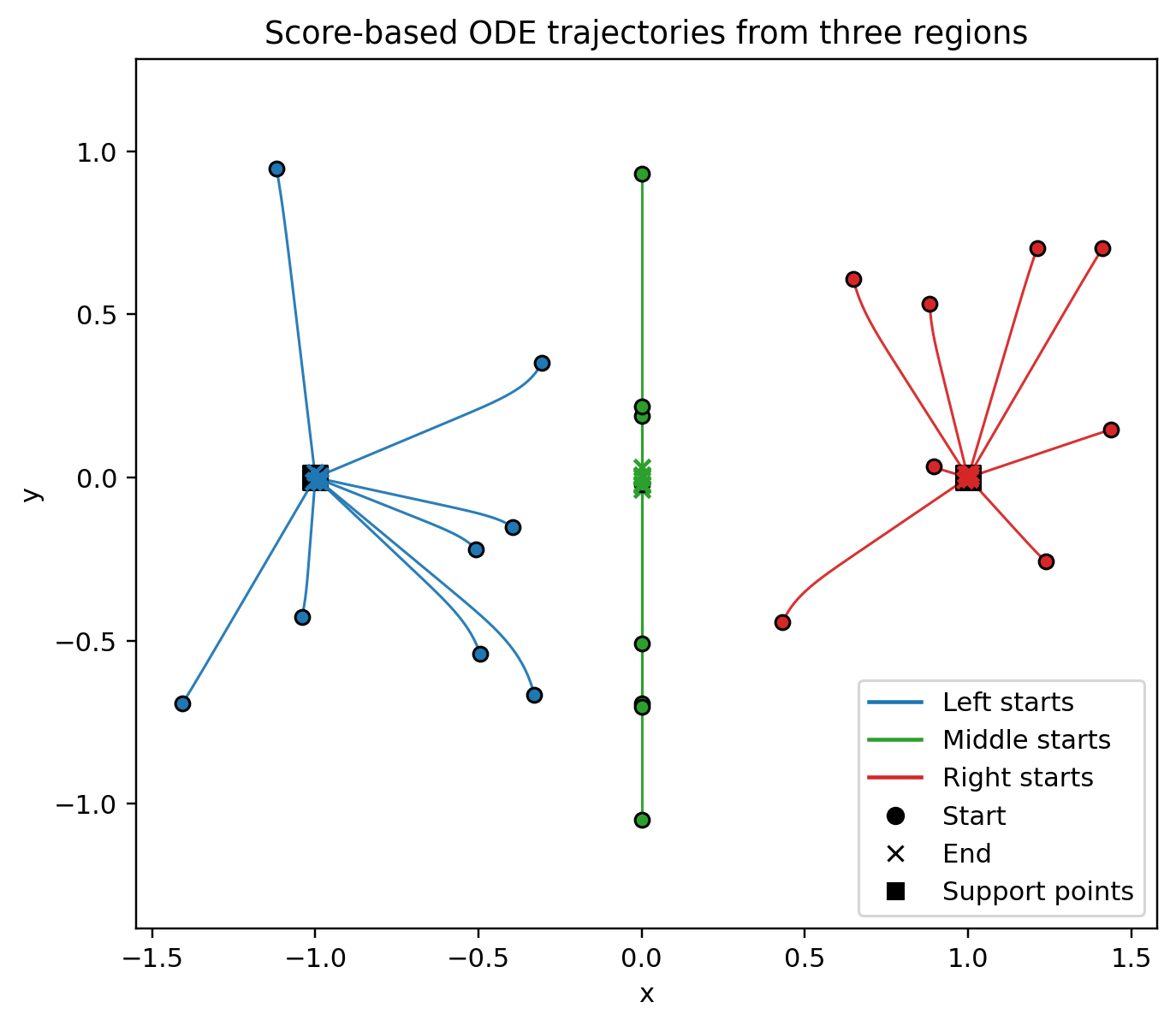}
  \caption{Trajectories of the score-based ODE ($\epsilon=0$) with $u_0 = \tfrac12\big(\delta_{(-1,0)} + \delta_{(1,0)}\big)$. 
  Initial points with $x<0$ converge to $(-1,0)$, those with $x>0$ converge to $(1,0)$, and points with $x=0$ converge to the barycenter $(0,0)$.}
  \label{fig:convergence}
\end{figure}

\subsection{Trajectory convergence rate in the ODE setting}
Theorem~\ref{thm:support} gives concentration and a $W_2$ rate for the generated laws. Here, we study the \textit{quantitative} concentration rate for the deterministic case of the generative flow \eqref{intro_eq:generation}, i.e., $\epsilon =0$. The resulting ODE is
\begin{equation}\label{eq:ODE}
\begin{cases}
\displaystyle \frac{d  X_t}{dt} = - s(X_t,\,t), & t\in(0,T),\\[0.6em]
 X_T = x_T,
\end{cases}
\end{equation}
where \(s\) is the score function \eqref{eq:score} and $x_T$ is any point in $\R^d$.

\Needspace{14\baselineskip}
\begin{thm}[Trajectory convergence]\label{thm:convergence_rate} 
Let $u_0\in\mathcal P(\R^d)$ have compact support
$S=\operatorname{supp}(u_0)$, and let $X_t$ solve
\eqref{eq:ODE} with terminal condition $X_T=x_T\in\R^d$.
Then the following assertions hold:
\begin{enumerate}
 \item[\emph{(i)}] 
 
For Lebesgue-almost every $x_T\in\R^d$, the trajectory $X_t$ has a limit in $S$ as $t\to0^+$. In particular, 
    \begin{equation}\label{eq:liminf_conv}
        \lim_{t\to0^+}\operatorname{dist}(X_t,S)=0.
    \end{equation}
  
\item[\emph{(ii)}]   
    Let \(K = \operatorname{conv}(S)\). Then, for any $x_T\in \R^d$,
    \begin{equation}\label{eq:dist-bound}
        \operatorname{dist}\bigl(X_t, K\bigr) 
        \;\leq\;
        \operatorname{dist}\bigl(x_T, K\bigr)\,T^{-1/2}\,\sqrt{t}, \qquad
     t\in (0,T].
    \end{equation}

   \item[\emph{(iii)}]
  {If $S=\{y_1,\ldots,y_N\}$ is finite, then for
Lebesgue-almost every $x_T\in\R^d$, there exist
$i\in\{1,\ldots,N\}$ and $C=C(u_0,x_T,T)>0$ such that}
\begin{equation}\label{eq:conv_1}
    |X_t-y_i|\leq C\sqrt t,
    \qquad 0<t\leq T.
\end{equation}
       
\end{enumerate}
\end{thm}
\begin{proof}
The proof is presented in Section~\ref{sec:proof-trajectories}.
\end{proof}

\begin{rem}
Assertion~\emph{(i)} concerns Lebesgue-almost every terminal point,
whereas \emph{(ii)} holds for every terminal point but only controls
the distance to the convex hull of the data support.
The symmetry example in Figure~\ref{fig:convergence} shows that
exceptional initializations may converge outside $S$.

A geometric description of such exceptional limits is developed
in~\cite{liuzuazua2026guided}. For
\[
    f_S(x)=\operatorname{dist}(x,S)^2,
    \qquad
    \Pi_S(x)=\argmin_{y\in S}|x-y|,
\]
one has
\[
    \partial^C f_S(x)
    =2\bigl(x-\operatorname{conv}\Pi_S(x)\bigr),
\]
where $\partial^C$ denotes the Clarke generalized gradient.
The corresponding equilibrium set is therefore
\[
    \mathcal E_S
    =\{x\in\R^d:x\in\operatorname{conv}\Pi_S(x)\}.
\]
The companion work shows that, for compact supports, every
accumulation point of a deterministic trajectory belongs to
$\mathcal E_S$.

The set $\mathcal E_S$ may strictly contain $S$. For instance, if
$S$ is a sphere centered at $a$, then
\[
    \mathcal E_S=S\cup\{a\}.
\]
For the uniform measure on the sphere, symmetry gives $s(a,t)=0$,
so the trajectory initialized at $a$ remains there. Thus the center
provides a simple example of an exceptional limit outside the
data support.

When $S$ is finite, assertion~\emph{(iii)} strengthens
\emph{(i)} by providing the quantitative rate
$\mathcal O(\sqrt t)$ toward the limiting data point.
The constant depends on $u_0$, the terminal point $x_T$, and the
time horizon $T$; no estimate uniform over all data configurations
or terminal points is asserted.
For general compact supports, Theorem~\ref{thm:support} provides
the corresponding square-root rate at the level of probability laws.
\end{rem}

\subsection{The explicit case \texorpdfstring{$\epsilon=1$}{epsilon=1}}
\label{sec:explicit}

Theorem~\ref{thm:support} describes the support and convergence of the
limiting law, but does not in general identify its probability weights.
When $\epsilon=1$, the quotient $v/u$ satisfies a backward heat equation,
which yields an explicit representation of the generated law and makes
the effect of the terminal initialization transparent.

\begin{prop}\label{prop:explicit-endpoint}
Assume the hypotheses of Theorem~\ref{thm:support} and let $\epsilon=1$.
Then, for $0<t<T$,
\begin{equation}\label{eq:explicit-solution}
 v(x,t)=u(x,t)\int_{\R^d}
       \frac{G_{T-t}(x-y)}{u(y,T)}\,dv_T(y).
\end{equation}
Moreover, the limiting law $v_0^*$ of Theorem~\ref{thm:support} is
given by
\begin{equation}\label{eq:explicit-endpoint}
 dv_0^*(x)=a_{v_T}(x)\,du_0(x),
 \qquad
 a_{v_T}(x):=
 \int_{\R^d}\frac{G_T(x-y)}{u(y,T)}\,dv_T(y).
\end{equation}
One has $0<a_{v_T}<\infty$ $u_0$-almost everywhere and
\[
    \int_{\R^d}a_{v_T}\,du_0=1.
\]
Hence $v_0^*$ and $u_0$ are equivalent.
\end{prop}
\nopagebreak[4]
\begin{proof}
The proof is given in Appendix~\ref{app:explicit}.
\end{proof}

The explicit formula also allows us to quantify the effect of the
terminal horizon for Gaussian initialization. Let
\begin{equation}\label{eq:data-moments}
 m_0:=\int_{\R^d}x\,du_0(x),
 \qquad
 \Sigma_0:=\int_{\R^d}(x-m_0)(x-m_0)^\top\,du_0(x).
\end{equation}

\begin{prop}\label{prop:moment-expansion}
Let $u_0\in\mathcal P(\R^d)$ have compact support $S$.
For $v_T=G_T$, write $a_T=a_{v_T}$. Then, uniformly for $x\in S$,
\begin{equation}\label{eq:moment-expansion}
 a_T(x)
 =1+\frac{\langle m_0,m_0-x\rangle}{2T}
   +\frac{B_0(x)}{8T^2}
   +\mathcal O(T^{-3}),
 \qquad T\to\infty,
\end{equation}
where
\begin{equation}\label{eq:moment-coefficient}
 B_0(x)
 =\langle m_0,x-m_0\rangle^2
  +\operatorname{tr}(\Sigma_0^2)
  -\langle\Sigma_0(x-2m_0),x-2m_0\rangle.
\end{equation}
If instead $v_T=G_T(\cdot-m_0)$ and
$a_T^{\mathrm c}=a_{v_T}$, then, uniformly for $x\in S$,
\begin{equation}\label{eq:centered-moment-expansion}
 a_T^{\mathrm c}(x)
 =1+\frac{
        \operatorname{tr}(\Sigma_0^2)
        -\langle\Sigma_0(x-m_0),x-m_0\rangle}
        {8T^2}
  +\mathcal O(T^{-3}),
 \qquad T\to\infty.
\end{equation}
\end{prop}
\nopagebreak[4]
\begin{proof}
The proof is given in Appendix~\ref{app:explicit}.
\end{proof}

\begin{rem}[Terminal initialization and limiting weights]
\label{rem:explicit-interpretation}
For $\epsilon=1$, under the preceding assumptions, the exact score
fixes the endpoint support, while the terminal law may change the
probability weights through
$dv_0^*=a_{v_T}\,du_0$. Proposition~\ref{prop:moment-expansion} shows
that for both Gaussian initializations considered above,
\[
    a_{v_T}(x)\longrightarrow1
    \quad\text{uniformly for }x\in S,
    \qquad T\to\infty.
\]
Thus the original data weights are recovered in the large-horizon
limit. Centering the Gaussian at the data barycenter removes the
first-order correction, while the second-order term retains information
on the covariance of $u_0$.
\end{rem}

\section{Score approximation, learning, and regularization}\label{sec:discussion}
The preceding results concern the exact score. We now turn to its
approximation by neural networks, the stability of the resulting
generative equation, and the role of regularization.

\subsection{Implicit regularization by neural networks}
We now turn to the training procedure of diffusion models and highlight how neural networks implicitly regularize the learned score function.

Assume we are given samples $\{y_i\}_{i=1}^N$ drawn from an unknown data distribution $u_0$.
Rather than directly constructing the empirical score function associated with these samples, one typically trains a neural network to approximate the (unknown) true score function $s(x,t)=\nabla_x\log u(x,t)$.
We write $s_\theta(x,t)$ for the neural network parameterized by
$\theta\in\R^P$. In this explanatory subsection, the time-integrated
losses on $(0,T)$ are written formally; when they diverge at zero,
a positive lower cutoff is needed, as discussed below. This leads to the standard score-matching
formulation~\cite{vincent2011connection}:
\begin{equation}\label{intro_pb:score-matching}
    \inf_{\theta \in \mathbb{R}^P} 
    \int_0^T \!\! \int_{\R^d}
    u(x,t)\,\big| s(x,t) - s_\theta(x,t)\big|^2\,dx\,dt,
\end{equation}
A motivation for this score-matching problem is provided by
Proposition~\ref{prop:score-stability} below, which quantifies the
relative-entropy error caused by replacing $s$ with $s_\theta$.

However, the optimization problem \eqref{intro_pb:score-matching} is not directly tractable in practice, as the true data distribution $u_0$, and hence the associated solution $u$ and score function $s$, are unknown.

The Gaussian convolution gives
\[
 u(x,t)s(x,t) =\nabla u(x,t) = \nabla G_{t} * u_0 (x) = \int_{\R^d} \frac{y-x}{2t} G_t(x-y) du_0(y),
\]
where $G_t$ is the Gaussian heat kernel. At each positive time, this identity gives the following formal
reformulation of \eqref{intro_pb:score-matching}, up to a
$\theta$-independent term:
\begin{equation}\label{intro_pb:score-matching-equi}
     \inf_{\theta \in \mathbb{R}^P} 
     \int_{\R^d}\!\! \int_0^T \!\! \int_{\R^d}
     \left| s_\theta(x,t) - \frac{y - x}{2t} \right|^2 
     G_t(x - y)\,dx\,dt\,du_0(y).
\end{equation}

In practice, we only have access to a finite collection of samples $\{y_k\}_{k=1}^N$ drawn from the initial distribution $u_0$.

By replacing the integral with respect to $u_0$ by the empirical average over the observed samples, we recover the classical empirical score-matching objective~\cite{vincent2011connection,song2021scorebased}:
\begin{equation}\label{pb:score-matching}
     \inf_{\theta \in \mathbb{R}^P} 
     \frac{1}{N}\sum_{k=1}^N \int_0^T \!\! \int_{\R^d}
     \left| s_\theta(x,t) - \frac{y_k - x}{2t} \right|^2 
     G_t(x - y_k)\,dx\,dt.
\end{equation}

Moreover, by performing the change of variables $z=(x-y_k)/(2t)$,
one obtains a score-parameterized form of the DDPM
loss~\cite{ho2020denoising}:
\begin{equation}\label{pb:score-matching-DDPM}
    \inf_{\theta\in \mathbb{R}^P} \frac{1}{N} \sum_{k=1}^N 
    \int_0^T \int_{\R^d} 
    \left| s_\theta(2tz+y_k,t) + z \right|^2 
    G_{(4t)^{-1}}(z)\,dz\,dt.
\end{equation}
Here, each clean data sample $y_k$ is perturbed to $x=2tz+y_k$, where
$z$ has density $G_{(4t)^{-1}}$. The target $-z$ is the negative of the
scaled noise $(x-y_k)/(2t)$. Equivalently, setting
$\zeta=\sqrt{2t}\,z$ gives standard Gaussian noise with density
$G_{1/2}$; the usual noise-prediction parameterization then carries the
time weight $1/(2t)$. Thus changing that weight changes the objective.

Another classical reformulation is the implicit score-matching loss introduced in~\cite{hyvarinen2005estimation}.
It can be derived by applying integration by parts to the original score-matching problem~\eqref{intro_pb:score-matching}.
In the finite-sample setting, this loss takes the form,
up to a parameter-independent constant,
\begin{equation}\label{pb:score-matching-h}
    \inf_{\theta \in \mathbb{R}^P} \frac{1}{N}\sum_{k=1}^N
    \int_0^T  \int_{\R^d} 
        \Big( |s_\theta(x,t) |^2 
        + 2\,\operatorname{div}_x\, s_\theta(x,t) \Big)\,
        G_t(x-y_k)\,dx\,dt.
\end{equation}
At each positive time, Gaussian decay justifies these integrations by
parts for differentiable scores of at most linear growth with bounded
first derivatives.

At each positive time, if there were no constraints on $s_\theta$, the
empirical score function~\eqref{eq:score_empirical} would minimize the
corresponding denoising objective by satisfying its first-order optimality
condition; see also~\cite{gu2023on,yi2023generalization,baptista2025memorization}.
By Theorem~\ref{thm:convergence_rate}, for Lebesgue-almost every
terminal point its deterministic trajectory converges to a training point. Restricting the score class can change this behavior,
but a small empirical fitting error alone does not control the
population error in Proposition~\ref{prop:score-stability}.

\Needspace{14\baselineskip}
\subsection{Stability under score approximation}\label{sec:score-stability}

The score-matching error has a direct interpretation in the generative
equation. The next result separates the terminal mismatch from the
accumulated score error in continuous time.
\begin{prop}\label{rem:score-matching}
\label{prop:score-stability}
Let $T>0$, $\epsilon>0$, and let $u_0\in\mathcal P(\R^d)$
have compact support. Suppose that $s_\theta$ is smooth and
globally Lipschitz in space, with at most linear growth,
uniformly on every $[\tau,T]$, $0<\tau<T$.
Under these assumptions, the positive-time well-posedness
argument of Proposition~\ref{prop:wellposed} applies with $s$ replaced
by $s_\theta$. Let $v_\theta$ denote the resulting unique narrowly
continuous weak probability solution of \eqref{eq:FP}, with score
$s_\theta$ and terminal probability density $v_{\theta,T}$ satisfying
$\mathrm{KL}(u(T)\|v_{\theta,T})<\infty$.
Then, for $0<t<T$,
\begin{equation}\label{eq:score-mismatch-bound}
 \mathrm{KL}(u(t)\|v_\theta(t))
 \leq \mathrm{KL}(u(T)\|v_{\theta,T})
 +\frac{(1+\epsilon)^2}{4\epsilon}
     \int_t^T\!\int_{\R^d}|s-s_\theta|^2u\,dx\,dr.
\end{equation}
\end{prop}
\begin{proof}
The proof is presented in Section~\ref{sec:proof-score}.
\end{proof}

\begin{rem}\label{rem:score-scope}
For $\epsilon=1$, \eqref{eq:score-mismatch-bound} has the heat-flow
form of the likelihood-weighted estimate in
\cite[Thm.~1]{song2021maximum}. Here we obtain it within the present
PDE framework by entropy truncation. The reference law in \eqref{eq:score-mismatch-bound} is always the
heat flow $u(t)$. Thus, if $v_{\theta,T}=u(T)$, the terminal mismatch
vanishes and the estimate measures solely the accumulated score error.
For a different terminal initialization, however, $u(t)$ is generally
not the exact-score generative solution issued from $v_{\theta,T}$.

The direction of the relative entropy is also important. Even if
$v_\theta(t)$ converges narrowly to $v_{\theta,0}$ as $t\to0^+$ and
the right-hand side remains bounded, the estimate controls
$\mathrm{KL}(u_0\|v_{\theta,0})$. In particular, finite endpoint relative entropy implies
$u_0\ll v_{\theta,0}$ and does not exclude additional generated mass
outside $\operatorname{supp}(u_0)$. In particular,
\eqref{eq:score-mismatch-bound} does not imply preservation of the
data support under score approximation.
\end{rem}

\subsection{Other regularization techniques}
The Li--Yau bound $\operatorname{div}s\geq-d/(2t)$ becomes singular as
$t\to0^+$. The following strategies are motivated by the preceding
analysis.
\begin{enumerate}
    \item \textit{Early-stopping strategy.}
    When using the empirical score function to generate new data, one can halt the reverse process slightly before $t=0$, at a time $0<t_{\min}<T$.
    {On $[t_{\min},T]$, the score divergence is bounded below by
    $-d/(2t_{\min})$.}
    This approach has been mentioned in~\cite{song2021scorebased,benton2024nearly,li2024good}; see also~\cite{debortoli2022manifold}.
    
Theorem~\ref{thm:support} bounds the $W_2$ distance between the stopped
    law and its exact-score endpoint by a constant times $\sqrt{t_{\min}}$.
    Theorem~\ref{thm:convergence_rate} gives a corresponding empirical
    trajectory-wise bound. These estimates quantify a stopping discrepancy,
    not an optimal stopping rule or a diversity guarantee.

\medskip
    \item \textit{Regularization in the loss function.}
    On the same positive-time interval, another approach is to add a penalty term based on the divergence of the learned score function to the score-matching loss~\eqref{pb:score-matching}, yielding:
    \begin{equation}\label{pb:score-matching-penalty}
     \inf_{\theta \in \mathbb{R}^P} 
     \frac{1}{N}\sum_{k=1}^N \int_{t_{\min}}^T \!\! \int_{\R^d}
     \left(
     \left| s_\theta(x,t) - \frac{y_k - x}{2t} \right|^2 
     + \lambda \, \big( \operatorname{div}_x\, s_\theta(x,t) \big)^2 
     \right) 
     G_t(x - y_k)\,dx\,dt,
    \end{equation}
    where $\lambda\geq0$ is a hyperparameter controlling the strength of the regularization.
   
Related regularization ideas appear in
\cite{kingma2010regularized,baptista2025memorization}, although
\eqref{pb:score-matching-penalty} is not intended to reproduce those
schemes. Our $L^p$ analysis instead highlights the negative part of
the score divergence as the quantity directly affecting stability.
This provides a motivation for divergence-based regularization, while
leaving questions of sample quality and optimal penalty design beyond
the scope of the present work.

\end{enumerate}

\section{Proofs of main results}\label{sec:proof}
Throughout this section, we denote by $u$ the solution of the heat
 equation~\eqref{eq:heat}, and by $v$ the solution of the backward FP
 equation~\eqref{eq:FP}.

\subsection{Heat-flow estimates and positive-time well-posedness}
\label{sec:heat-estimates}\label{sec:proof-wellposed}
We collect the estimates for the reference heat flow that are used
throughout the proofs. They also verify the hypotheses of the
classical well-posedness results applied below.

\begin{lem}[A priori estimates for the heat equation]
\label{lem:gaussian-bounds}\label{lem:score-bounds}
Let $u_0\in\mathcal P(\R^d)$, $u(t)=G_t*u_0$, and
$s=\nabla\log u$. Then $u$ is strictly positive, $u$ and $s$ are
smooth at positive times, and
\begin{equation}\label{eq:Li-Yau}
    \operatorname{div}s(x,t)\geq-\frac{d}{2t},
    \qquad (x,t)\in Q.
\end{equation}
If, in addition, $\operatorname{supp}(u_0)\subset B_R(0)$ for some
$R>0$, the following estimates hold, where $a_+=\max\{a,0\}$.
\begin{enumerate}
\item[\emph{(i)}] For every $(x,t)\in Q$,

\begin{equation}\label{eq:gaussian-bounds}
    (4\pi t)^{-d/2}\exp\!\left(-\frac{(|x|+R)^2}{4t}\right)
    \leq u(x,t)
    \leq(4\pi t)^{-d/2}
             \exp\!\left(-\frac{(|x|-R)_+^2}{4t}\right).
\end{equation}

\item[\emph{(ii)}] For every $(x,t)\in Q$,

\begin{equation}\label{eq:score-linear-growth}
    |s(x,t)|\leq\frac{|x|+R}{2t},
\end{equation}

and
\begin{equation}\label{eq:hessian-bound}
    -\frac1{2t}I_d\preceq\nabla_xs(x,t)
    \preceq\left(-\frac1{2t}+\frac{R^2}{4t^2}\right)I_d.
\end{equation}
\end{enumerate}
\end{lem}
\begin{proof}
Positivity and smoothness follow from the Gaussian convolution.
For $(x,t)\in Q$, define the probability measure and its mean

\begin{equation}\label{eq:posterior-mean}
    d\pi_{x,t}(y):=\frac{G_t(x-y)}{u(x,t)}\,du_0(y),
    \qquad m(x,t):=\int_{\R^d}y\,d\pi_{x,t}(y).
\end{equation}
Differentiating the convolution gives
\begin{equation}\label{eq:score-barycenter}
    s(x,t)=\frac{m(x,t)-x}{2t}.
\end{equation}
A second differentiation yields
\[
    \nabla_xs(x,t)=-\frac{I_d}{2t}+\frac{\Sigma(x,t)}{4t^2},
    \qquad
    \Sigma(x,t)=\int_{\R^d}(y-m(x,t))(y-m(x,t))^\top\,d\pi_{x,t}(y).
\]
These differentiated integrals are finite even without moments of $u_0$,
since a polynomial times a Gaussian is bounded at every fixed positive
time, locally uniformly in $(x,t)$. The covariance matrix is
nonnegative. Taking the trace proves \eqref{eq:Li-Yau}.

Assume now that $\operatorname{supp}(u_0)\subset B_R(0)$.

\medskip
\noindent\textbf{Step 1 (Proof of Gaussian bounds).}
Since \(\operatorname{supp}(u_0) \subset B_R(0)\), we have that for any \(x \in \R^d\),
\[
u(x,t) = \int_{B_R(0)} G_t(x - y) d u_0(y) .
\]
For all \(y \in B_R(0)\), the triangle inequality implies

\[
(|x|-R)_+\leq|x-y|\leq|x|+R,
\]
Therefore,
\[
\exp\!\left(-\frac{(|x|+R)^2}{4t}\right) 
\leq \exp\!\left(-\frac{|x - y|^2}{4t}\right)
\leq \exp\!\left(-\frac{(|x|-R)_+^2}{4t}\right).
\]

By the definition of $G_t$ and the fact that  \(\int_{B_R(0)} d u_0(y) = 1\), the estimate \eqref{eq:gaussian-bounds} follows.

\medskip
\noindent\textbf{Step 2 (Score and Hessian bounds).}

The measure $\pi_{x,t}$ is supported in $B_R(0)$, so $|m(x,t)|\leq R$.
Formula~\eqref{eq:score-barycenter} gives \eqref{eq:score-linear-growth}.
For any unit vector $e\in\R^d$, the covariance satisfies
\[
    0\leq \langle\Sigma(x,t)e,e\rangle
    =\int_{\R^d}\langle e,y\rangle^2\,d\pi_{x,t}(y)-\langle e,m(x,t)\rangle^2
    \leq R^2.
\]
Thus $0\preceq\Sigma(x,t)\preceq R^2I_d$. Substitution into the
covariance formula above yields \eqref{eq:hessian-bound}.
\end{proof}

\begin{proof}[Proof of Proposition~\ref{prop:wellposed}]
Fix $0<\tau<T$. By Lemma~\ref{lem:gaussian-bounds}, the score
$s=\nabla\log u$ is smooth and satisfies
\begin{equation}\label{eq:positive-time-score-bounds}
    |s(x,t)|\leq C_\tau(1+|x|),
    \qquad
    \|\nabla_xs(\cdot,t)\|_{L^\infty(\R^d)}
    \leq C_\tau,
    \qquad \tau\leq t\leq T.
\end{equation}
Thus the coefficients of \eqref{eq:FP} are regular on every
positive-time strip, although the corresponding bounds may
diverge as $\tau\to0^+$.

For the application of the classical Cauchy theory, introduce
the increasing time variable
\[
    r=T-t,\qquad
    \rho(r)=v(T-r),\qquad
    b(x,r)=(1+\epsilon)s(x,T-r).
\]
Then \eqref{eq:FP} becomes
\begin{equation}\label{eq:forward-FP}
    \partial_r\rho
    =
    \epsilon\Delta\rho-\operatorname{div}(b\rho),
    \qquad
    \rho(0)=v_T,
    \qquad
    0<r<T-\tau.
\end{equation}

\medskip
\noindent\textbf{Step 1 (Existence and uniqueness).}
Suppose first that $\epsilon>0$. The diffusion matrix
$A=\epsilon I_d$ is constant and positive definite, while
$b$ is smooth and has at most linear growth. Hence
\cite[Thm.~6.6.4]{bogachev2022fokker} gives a nonnegative
weak solution of \eqref{eq:forward-FP}. Moreover, for
$V(x)=1+|x|^2$ and 
\[
    \mathcal L_r\varphi
    =\epsilon\Delta\varphi
      +\langle b(\cdot,r),\nabla\varphi\rangle,
\]
we have
\[
    \mathcal L_rV(x)
    =
    2\epsilon d+2\langle b(x,r),x\rangle
    \leq C_\tau V(x).
\]
By~\cite[Cor.~6.6.6]{bogachev2022fokker}, the solution
conserves mass. We take its probability-valued representative;
the continuity of this representative is specified below.

Moreover,
\[
    \frac{\|A\|_{\mathrm{op}}}{1+|x|^2}
    +
    \frac{|b(x,r)|}{1+|x|}
    \leq C_\tau.
\]
Here $\|A\|_{\mathrm{op}}$ denotes the operator norm.
The local coefficient conditions and the coefficient-integrability
assumption of \cite[Thm.~9.4.3]{bogachev2022fokker} are therefore
satisfied, and uniqueness follows in the weak probability class.
The weak formulation gives continuity against compactly supported
smooth tests. Conservation of mass then upgrades this representative
to narrow continuity, including its prescribed initial trace.
Returning to the original time variable gives
\[
    v\in C\bigl([\tau,T];\mathcal P(\R^d)\bigr),
    \qquad v(T)=v_T.
\]
Since the diffusion matrix is nondegenerate and the coefficients
are smooth on the strip, standard interior parabolic regularity
and positivity imply that $v$ has a smooth, strictly positive
density for $\tau<t<T$; see
\cite[Sects.~6.3--6.4 and~8.1]{bogachev2022fokker}.

When $\epsilon=0$, \eqref{eq:forward-FP} reduces to a transport
equation. By \eqref{eq:positive-time-score-bounds}, $b$ is
globally Lipschitz in space.
The characteristic representation therefore gives
a unique narrowly continuous weak probability solution;
see~\cite[Sect.~9.8(iv)]{bogachev2022fokker}.
This proves part~\emph{(i)} on $[\tau,T]$.

\medskip
\noindent\textbf{Step 2 (Wasserstein continuity).}
Assume now that $v_T\in\mathcal P_2(\R^d)$. The same Lyapunov
estimate with $V(x)=1+|x|^2$ gives propagation of the second
moment on every positive-time strip; see
\cite[Sect.~7.1]{bogachev2022fokker}. In particular,
\begin{equation}\label{eq:positive-time-moment}
    \sup_{t\in[\tau,T]}\int_{\R^d}|x|^2\,dv(t)(x)
    \leq C_\tau\left(1+\int_{\R^d}|x|^2\,dv_T(x)\right).
\end{equation}
Testing with $\varphi_L(x)=|x|^2\chi(x/L)$, where
$\chi\in C_c^\infty(\R^d)$ equals one near the origin,
and removing the cutoff, using linear growth and the preceding moment bound,
gives, for $\tau\leq t_1<t_2\leq T$,
\[
 \int_{\R^d}|x|^2\,dv(t_2)(x)-\int_{\R^d}|x|^2\,dv(t_1)(x)
 =\int_{t_1}^{t_2}\!\int_{\R^d}
       \bigl(-2\epsilon d-2(1+\epsilon)\langle s(x,r),x\rangle\bigr)
       \,dv(r)(x)\,dr.
\]
Indeed, $|\nabla\varphi_L|\leq C(1+|x|)$ and
$|\Delta\varphi_L|\leq C$, uniformly in $L\geq1$.
The integrands are therefore dominated by $C_\tau(1+|x|^2)$;
\eqref{eq:positive-time-moment} justifies the space--time
passage to the limit. The identity shows that the second
moment is absolutely continuous on $[\tau,T]$.
Together with narrow continuity, the characterization of
Wasserstein convergence
\cite[Thm.~6.9]{villani2008optimal} yields
\[
    v\in
    C\bigl([\tau,T];\mathcal P_2(\R^d),W_2\bigr).
\]

Finally, uniqueness implies that the solutions obtained on
different positive-time strips coincide on their overlaps.
Since $\tau>0$ is arbitrary, the preceding conclusions hold
on $(0,T]$, proving both assertions of the proposition.
\end{proof}

\subsection{The Li--Yau bound and the energy estimate}\label{sec:proof-Lp}

We explain the role of the Li--Yau inequality in the energy estimate.
For this calculation, let $1<p<\infty$ and assume that $v$ is a smooth,
strictly positive solution of \eqref{eq:FP}, with sufficient decay at
infinity to justify the differentiations and integrations by parts below.
We also assume that $v(t)\to v_T$ in $L^p$ as $t\to T^-$.
These additional assumptions are used only in this subsection.
Appendix~\ref{app:wellposed} proves Theorem~\ref{thm:energy_score} under the stated hypotheses.

\medskip
\noindent\textbf{Step 1 (Regularity of the score function).}  
Recall that \(Q = \R^d \times (0, T]\) and that
\begin{equation*}
  s(x,t) = \nabla \log u(x,t) = \frac{\nabla u(x,t)}{u(x,t)}, 
  \qquad \forall\, (x,t) \in Q.
\end{equation*}
By the convolution representation of \(u\), the score function \(s\) is well-defined and belongs to \(\mathcal{C}^{\infty}(Q)\).  
The one-sided divergence bound is \eqref{eq:Li-Yau}, established in
Lemma~\ref{lem:gaussian-bounds}; it is the Euclidean Li--Yau
inequality recalled in Section~\ref{sec:preliminary-stability}.

\medskip
\noindent\textbf{Step 2 (Energy identity).}
Multiply equation \eqref{eq:FP} by \( v^{p-1} \) and integrate over \( \R^d \):
\[
\int_{\R^d} \left(
v^{p-1} \partial_t v
+ \epsilon\, v^{p-1} \Delta v
- (1+\epsilon)\, v^{p-1} \, \operatorname{div}(s\,v)
\right) \,dx = 0.
\]
Each term is handled as follows:

\begin{itemize}
    \item Time derivative:
\[
\int_{\R^d} v^{p-1} \partial_t v \, \,dx
= \frac{1}{p} \frac{d}{dt} \| v \|_{L^p}^p.
\]
\item Diffusion term:
\[
\int_{\R^d} v^{p-1} \Delta v\, \,dx
= - (p-1) \int_{\R^d} v^{p-2} | \nabla v |^2 \,dx.
\]
\item Advection term:
\begin{align*}
\int_{\R^d} v^{p-1} \, \operatorname{div}(s\,v) \, \,dx
&= - \int_{\R^d} \big\langle \nabla(v^{p-1}),\, v\,s \big\rangle \, \,dx \\
&= - (p-1) \int_{\R^d} v^{p-2} \big\langle \nabla v,\, v\,s \big\rangle \, \,dx \\
&= - (p-1) \int_{\R^d} v^{p-1} \big\langle \nabla v,\, s \big\rangle \, \,dx \\
&= - \frac{p-1}{p} \int_{\R^d} \big\langle \nabla(v^p),\, s \big\rangle \, \,dx
= \frac{p-1}{p} \int_{\R^d} v^p \, \operatorname{div}(s) \, \,dx.
\end{align*}

\end{itemize}
Putting all terms together yields the energy identity:
\begin{equation}\label{eq:energy_identity}
\frac{1}{p} \frac{d}{dt} \| v \|_{L^p}^p
- \epsilon (p-1) \int_{\R^d} v^{p-2} | \nabla v |^2 \,dx
- (1+\epsilon) \frac{p-1}{p} \int_{\R^d} v^p \, \operatorname{div}(s) \, \,dx = 0.
\end{equation}

\medskip
\noindent\textbf{Step 3 (Energy estimate).} 
Now, using the lower bound on \(\operatorname{div}(s)\) from~\eqref{eq:Li-Yau}, 
we deduce from~\eqref{eq:energy_identity} the following differential inequality:
\begin{equation*}
    \begin{cases}
        \displaystyle \frac{d}{dt} \| v(t) \|_{L^p}^p 
        \;\geq\; -\,\frac{d(1+\epsilon)(p-1)}{2t}\, \| v(t) \|_{L^p}^p, 
        & \text{for a.e. } t \in (0, T), \\[0.8em]
        \|v(T)\|_{L^p}^p = \|v_T\|_{L^p}^p.
    \end{cases}
\end{equation*}
Since $v>0$, dividing by \(\|v(t)\|_{L^p}^p\) yields
\[
    \frac{d}{dt} \log \bigl( \|v(t)\|_{L^p}^p \bigr) 
    \;\geq\; -\,\frac{C}{2t},
    \qquad C := d(1+\epsilon)(p-1).
\]
Consequently,
\[
    \frac{d}{dt} \log \bigl( \|v(t)\|_{L^p} \bigr) 
    \;\geq\; -\,\frac{C}{2 p t}.
\]
Integrating this differential inequality from \(t\) to \(T\) gives
\[
    \log \bigl( \|v_T\|_{L^p} \bigr) 
    - \log \bigl( \|v(t)\|_{L^p} \bigr) 
    \;\geq\; -\,\frac{C}{2p} \log \!\left( \frac{T}{t} \right).
\]
Exponentiating both sides, we obtain
\[
    \|v(t)\|_{L^p} 
    \;\leq\; 
    \left( \frac{T}{t} \right)^{\!\frac{C}{2p}} \, \|v_T\|_{L^p}.
\]

This gives \eqref{eq:energy_score} under the additional assumptions of
this subsection. For a nonnegative probability solution, the case $p=1$
is simply conservation of mass. A rigorous proof under the original
hypotheses of Theorem~\ref{thm:energy_score}, without the auxiliary
smoothness, strict positivity, and decay assumptions used here, is given
in Appendix~\ref{app:wellposed}.

\begin{proof}[Proof of Remark~\ref{rem:Lp-stability}]
By scaling Proposition~\ref{prop:wellposed}, finite nonnegative
terminal data define nonnegative weak solutions with conserved
mass, unique in this class.
Set
\[
    w_T=v_{1,T}-v_{2,T},
    \qquad M=\|w_T\|_{L^1}.
\]
If $M=0$, the conclusion follows from uniqueness.
Otherwise, let $q^+$ and $q^-$ be the nonnegative solutions
with terminal data $(w_T)_+$ and $(w_T)_-$, respectively,
and set $q=q^++q^-$.

The nonnegative solutions $v_1+q^-$ and $v_2+q^+$ have the
same terminal datum, since
\[
    v_{1,T}+(w_T)_-
    =
    v_{2,T}+(w_T)_+
    =
    \max\{v_{1,T},v_{2,T}\}.
\]
After normalization by their common mass, uniqueness in
Proposition~\ref{prop:wellposed} gives
$v_1+q^-=v_2+q^+$. Consequently, for every $0<t\leq T$,
\[
    |v_1(t)-v_2(t)|
    =
    |q^+(t)-q^-(t)|
    \leq q(t)
    \qquad\text{a.e. in }\R^d.
\]
The function $q/M$ is the probability solution with terminal
density $|w_T|/M\in L^p(\R^d)$. Applying
Theorem~\ref{thm:energy_score} therefore yields
\[
    \|v_1(t)-v_2(t)\|_{L^p}
    \leq \|q(t)\|_{L^p}
    \leq
    \left(\frac{T}{t}\right)^{\frac{d(1+\epsilon)(p-1)}{2p}}
    \|w_T\|_{L^p}.
\]
This proves \eqref{eq:Lp-stability}.
\end{proof}

\subsection{Proof of Theorem \ref{thm:support}}\label{sec:proof-endpoint}
We use the heat-flow estimates in Lemma~\ref{lem:gaussian-bounds}
and the following two entropy lemmas. Their proofs are in
Section~\ref{sec:proof_lemmas}; Appendix~\ref{app:entropy-justification}
provides the localization and terminal-passage arguments.

\begin{lem}[Terminal bound]\label{lm:absolutely-continuous}
Under the hypotheses of Theorem~\ref{thm:support}, one has:
\[
    \mathrm{KL}\bigl(v_T \,\|\, u(T)\bigr) < \infty.
\]
Moreover, if $|\log v_T(x)|\leq C(1+|x|^q)$ for almost every $x\in\R^d$, for some $C,q>0$, then
\[
    \mathrm{KL}\bigl(u(T) \,\|\, v_T\bigr) < \infty.
\]
\end{lem}

\begin{proof}
The proof is presented in Section~\ref{sec:proof_lemmas}.
\end{proof}

\begin{lem}[Relative-entropy contraction and dissipation]\label{lem:KL}
Let $v$ solve \eqref{eq:FP} under the hypotheses of
Theorem~\ref{thm:support}. Write $H(t)=\mathrm{KL}(v(t)\|u(t))$.
Then, for $0<t_1\leq t_2\leq T$,
\[
 H(t_1)\leq H(t_2)\leq H(T)<\infty.
\]
If $\epsilon>0$, for every $0<t_1<t_2\leq T$,
\begin{equation}\label{KL:integrated}
 H(t_2)-H(t_1)=\epsilon\int_{t_1}^{t_2}\!\int_{\R^d}
       v\left|\nabla\log\frac vu\right|^2\,dx\,dt.
\end{equation}
In particular, $H$ is absolutely continuous on every $[\tau,T]$ and
\begin{equation}\label{KL:identity-1}
 \frac{d}{dt}\mathrm{KL}(v(t)\|u(t))
 =\epsilon\int_{\R^d}v\left|\nabla\log\frac vu\right|^2\,dx
 \quad\text{for a.e. }t\in(0,T).
\end{equation}
If also $\mathrm{KL}(u(T)\|v_T)<\infty$, the same statements hold with
$u$ and $v$ interchanged; in particular,
\begin{equation}\label{KL:identity-2}
 \frac{d}{dt}\mathrm{KL}(u(t)\|v(t))
 =\epsilon\int_{\R^d}u\left|\nabla\log\frac uv\right|^2\,dx
 \quad\text{for a.e. }t\in(0,T).
\end{equation}
For $\epsilon=0$, each finite relative entropy above is constant in time.
\end{lem}

\begin{proof}
The proof is presented in Section~\ref{sec:proof_lemmas}; its
localization and terminal-passage arguments are given in
Appendix~\ref{app:entropy-justification}.
\end{proof}

\begin{proof}[Proof of Theorem~\ref{thm:support}]
The solution is supplied by Proposition~\ref{prop:wellposed}\emph{(ii)}.
Recall that
\[
    K_T=\mathrm{KL}\bigl(v_T\|u(T)\bigr),
    \qquad
    C_T=2\sqrt{2K_T+d\log2}+\sqrt{\epsilon K_T}.
\]
By Lemma~\ref{lem:KL},
\begin{equation}\label{eq:KL-uniform}
    \mathrm{KL}\bigl(v(t)\|u(t)\bigr)
    \leq K_T,
    \qquad 0<t\leq T.
\end{equation}

\medskip
\noindent\textbf{Step 1 (Estimate of the score along the generated law).}
Recall the probability measure $\pi_{x,t}$ defined in
\eqref{eq:posterior-mean}:
\[
    d\pi_{x,t}(y)
    =\frac{G_t(x-y)}{u(x,t)}\,du_0(y).
\]
By the mean-shift identity \eqref{eq:score-barycenter},
\[
    s(x,t)
    =
    \int_{\R^d}\frac{y-x}{2t}\,d\pi_{x,t}(y).
\]
For each fixed $t>0$, the function
$z\mapsto e^{t|z|^2/2}$ is convex. Hence, by Jensen's inequality,
\begin{align*}
    e^{t|s(x,t)|^2/2}
    &=
    \exp\!\left(
        \frac{t}{2}
        \left|
            \int_{\R^d}\frac{y-x}{2t}\,d\pi_{x,t}(y)
        \right|^2
    \right)\\
    &\leq
    \int_{\R^d}
        e^{|x-y|^2/(8t)}\,d\pi_{x,t}(y) = \int_{\R^d}
        e^{|x-y|^2/(8t)}\,\frac{G_t(x-y)}{u(x,t)}\,du_0(y).
\end{align*}
Multiplying by $u(x,t)$, integrating in $x$, and using Tonelli's
theorem, we obtain
\begin{align*}
    \int_{\R^d}e^{t|s(x,t)|^2/2}u(x,t)\,dx
    &\leq
    \iint_{\R^d\times\R^d}
        e^{|x-y|^2/(8t)}G_t(x-y)\,dx\,du_0(y)\\
    &=
    (4\pi t)^{-d/2}
    \iint_{\R^d\times\R^d}
        e^{-|x-y|^2/(8t)}\,dx\,du_0(y)\\
    &=2^{d/2}.
\end{align*}

We transfer this estimate from $u(t)$ to $v(t)$ using the
Donsker--Varadhan variational inequality
\cite{donsker1975asymptotic}. For
\[
    \phi_n(x)
    :=
    \min\left\{\frac{t}{2}|s(x,t)|^2,n\right\},
\]
we have
\[
    \int_{\R^d}\phi_n(x)v(x,t)\,dx
    \leq
    \mathrm{KL}\bigl(v(t)\|u(t)\bigr)
    +\log\int_{\R^d}e^{\phi_n(x)}u(x,t)\,dx.
\]
Using \eqref{eq:KL-uniform} and the preceding exponential estimate,
\[
    \int_{\R^d}\phi_n(x)v(x,t)\,dx
    \leq K_T+\frac d2\log2.
\]
Letting $n\to\infty$ by monotone convergence gives
\begin{equation}\label{eq:score-generated-L2}
    \int_{\R^d}|s(x,t)|^2v(x,t)\,dx
    \leq
    \frac{2K_T+d\log2}{t}
    =:\frac{C_0^2}{t},
    \qquad
    C_0:=\sqrt{2K_T+d\log2}.
\end{equation}

\medskip
\noindent\textbf{Step 2 ($W_2$ convergence and rate).}
We rewrite the backward FP equation as a continuity equation.
Suppose first that $\epsilon>0$. Since $v$ is smooth and strictly
positive on $\mathring Q$,
\[
    \nabla v
    =
    v\nabla\log v
    =
    v\left(
        s+\nabla\log\frac vu
    \right).
\]
Therefore
\begin{align*}
    \epsilon\nabla v-(1+\epsilon)sv
    &=
    \epsilon v\left(
        s+\nabla\log\frac vu
    \right)
    -(1+\epsilon)sv\\
    &=
    v\left(
        -s+\epsilon\nabla\log\frac vu
    \right).
\end{align*}
When $\epsilon=0$, the flux is simply $-sv$. Thus, in both cases,
\begin{equation}\label{eq:current-velocity}
    \partial_t v+\operatorname{div}(b_vv)=0,
    \qquad
    b_v(\cdot,t)
    =
    \begin{cases}
        -s+\epsilon\nabla\log(v/u),&\epsilon>0,\\[0.3em]
        -s,&\epsilon=0,
    \end{cases}
\end{equation}
in the sense of distributions on $\mathring Q$.

Fix $0<t_1<t_2\leq T$. If $\epsilon>0$, set
\[
    I(t)
    :=
    \int_{\R^d}
        v(x,t)
        \left|
            \nabla\log\frac{v(x,t)}{u(x,t)}
        \right|^2\,dx.
\]
By the entropy dissipation identity \eqref{KL:integrated},
\begin{align*}
    \epsilon\int_{t_1}^{t_2}I(t)\,dt
    =
    \mathrm{KL}\bigl(v(t_2)\|u(t_2)\bigr)
    -
    \mathrm{KL}\bigl(v(t_1)\|u(t_1)\bigr)\leq K_T.
\end{align*}
By \eqref{eq:score-generated-L2} and the entropy dissipation estimate,
the velocity in \eqref{eq:current-velocity} satisfies
\[
    r\longmapsto\|b_v(\cdot,r)\|_{L^2(v(r))}
    \in L^1((t_1,t_2)).
\]
Indeed, for $\epsilon>0$,
\begin{align*}
    \int_{t_1}^{t_2}
        \|b_v(\cdot,r)\|_{L^2(v(r))}\,dr
    &\leq
    \int_{t_1}^{t_2}
        \|s(\cdot,r)\|_{L^2(v(r))}\,dr
    +\epsilon\int_{t_1}^{t_2}\sqrt{I(r)}\,dr\\
    &\leq
    2C_0\bigl(\sqrt{t_2}-\sqrt{t_1}\bigr)
    +\sqrt{\epsilon K_T(t_2-t_1)}
    <\infty.
\end{align*}
When $\epsilon=0$, only the first term is present.

Since $v$ is $W_2$-continuous on $[t_1,t_2]$ by
Proposition~\ref{prop:wellposed}\emph{(ii)}, the converse implication in
\cite[Thm.~8.3.1]{ambrosio2008gradient}, applied to the continuity
equation \eqref{eq:current-velocity}, yields metric absolute continuity
and
\begin{equation}\label{eq:W2-cauchy}
    W_2\bigl(v(t_1),v(t_2)\bigr)
    \leq\int_{t_1}^{t_2}
        \|b_v(\cdot,r)\|_{L^2(v(r))}\,dr \leq 
    2C_0\bigl(\sqrt{t_2}-\sqrt{t_1}\bigr)
    +\sqrt{\epsilon K_T(t_2-t_1)}.
\end{equation}
The right-hand side tends to zero as $t_1,t_2\to0^+$.
Hence $v(t)$ is Cauchy in $(\mathcal P_2(\R^d),W_2)$.
By completeness, there exists a unique
$v_0^*\in\mathcal P_2(\R^d)$ such that
\[
    v(t)\longrightarrow v_0^*
    \qquad\text{in }W_2
    \quad\text{as }t\to0^+.
\]
Fixing $0<t\leq T$, setting $t_2=t$, and letting
$t_1\to0^+$ in \eqref{eq:W2-cauchy}, we obtain
\[
    W_2\bigl(v(t),v_0^*\bigr)
    \leq
    \left(
        2C_0+\sqrt{\epsilon K_T}
    \right)\sqrt t
    =
    C_T\sqrt t.
\]
This proves \eqref{eq:W2-endpoint-rate} and gives a
$W_2$-continuous extension to $[0,T]$ by setting $v(0)=v_0^*$.

\medskip
\noindent\textbf{Step 3 (Identification of the endpoint).}
The convolution representation gives the coupling bound
\[
    W_2\bigl(u(t),u_0\bigr)^2
    \leq\iint_{\R^d\times\R^d}|x-y|^2G_t(x-y)\,dx\,du_0(y)
    =2dt.
\]
Thus $u(t)\to u_0$ in $W_2$ as $t\to0^+$.
Together with Step~2, this implies the narrow convergences
\[
    v(t)\rightharpoonup v_0^*,
    \qquad
    u(t)\rightharpoonup u_0.
\]
By the joint lower semicontinuity of relative entropy,
\[
    \mathrm{KL}\bigl(v_0^*\|u_0\bigr)
    \leq
    \liminf_{t\to0^+}
        \mathrm{KL}\bigl(v(t)\|u(t)\bigr)
    \leq K_T<\infty.
\]
Therefore
\[
    v_0^*\ll u_0,
    \qquad
    \operatorname{supp}(v_0^*)\subset
    \operatorname{supp}(u_0)=S.
\]

Assume now, in addition, that
\[
    |\log v_T(x)|
    \leq C(1+|x|^q)
    \qquad\text{for almost every }x\in\R^d.
\]
By Lemma~\ref{lm:absolutely-continuous},
$\mathrm{KL}(u(T)\|v_T)<\infty$, and Lemma~\ref{lem:KL} yields
\[
    \mathrm{KL}\bigl(u(t)\|v(t)\bigr)
    \leq
    \mathrm{KL}\bigl(u(T)\|v_T\bigr),
    \qquad 0<t\leq T.
\]
Applying joint lower semicontinuity once more gives
\[
    \mathrm{KL}\bigl(u_0\|v_0^*\bigr)
    \leq
    \liminf_{t\to0^+}
        \mathrm{KL}\bigl(u(t)\|v(t)\bigr)
    <\infty.
\]
Hence $u_0\ll v_0^*$. Together with the previous inclusion,
\[
    \operatorname{supp}(v_0^*)
    =
    \operatorname{supp}(u_0)
    =
    S.
\]

\medskip
\noindent\textbf{Step 4 (Concentration on the data support).}
Since $\operatorname{supp}(v_0^*)\subset S$, every coupling
$\pi$ of $v(t)$ and $v_0^*$ satisfies
$\pi(\R^d\times S)=1$. Therefore
\[
    \operatorname{dist}(x,S)\leq|x-y|
    \qquad
    \pi\text{-almost everywhere},
\]
and hence
\[
    \int_{\R^d}\operatorname{dist}(x,S)^2\,dv(t)(x)
    \leq
    \iint_{\R^d\times\R^d}|x-y|^2\,d\pi(x,y).
\]
Taking the infimum over all couplings yields
\[
    \left(
        \int_{\R^d}\operatorname{dist}(x,S)^2\,dv(t)(x)
    \right)^{1/2}
    \leq
    W_2\bigl(v(t),v_0^*\bigr)
    \leq C_T\sqrt t,
\]
which proves \eqref{eq:mean-square-support-rate}.
By Markov's inequality, for every $r>0$,
\[
    v(t)\bigl\{x:\operatorname{dist}(x,S)\geq r\bigr\}
    \leq
    \min\left\{1,\frac{C_T^2t}{r^2}\right\},
\]
which is \eqref{eq:concentration-tail}.

Finally, let $U$ be an open neighborhood of $S$.
Since $S$ is compact, there exists $r>0$ such that
\[
    \{x:\operatorname{dist}(x,S)<r\}\subset U.
\]
Consequently,
\[
    v(t)(U^c)
    \leq
    \frac{C_T^2t}{r^2}
    \longrightarrow0,
    \qquad t\to0^+,
\]
and therefore
\[
    v(t)(U)\longrightarrow1.
\]

For $\epsilon>0$, the FP--SDE correspondence gives
$v(t)=\operatorname{law}(X_t)$; for $\epsilon=0$ the same identity
follows from the characteristic representation. The conclusion follows.
\end{proof}

\subsection{Proof of Theorem \ref{thm:convergence_rate}}\label{sec:proof-trajectories}
For convenience, we recall the generation ODE \eqref{eq:ODE}:
\begin{equation}\label{eq:ODE_proof}
\begin{cases}
\displaystyle \frac{dX_t}{dt} = -\,s(X_t,\,t), & t \in (0,T),\\[0.6em]
X_T = x_T \in \R^d.
\end{cases}
\end{equation}

The proof combines a finite-mean-length argument with quantitative
estimates from the Gaussian mean-shift formula. Its three steps
correspond to assertions \emph{(i)}--\emph{(iii)} of the theorem.

\medskip

\noindent\textbf{Step 1 (Almost-everywhere convergence).}
We use the reference terminal density $u(T)$ to control the total length
of the deterministic trajectories. The heat flow itself solves
\eqref{eq:FP}; uniqueness and the characteristic representation therefore
give
\[
 u(t)=(\Phi_{t,T})_\#u(T),
\]
where $\Phi_{t,T}(x)$ is the trajectory at time $t$ starting from $x$
at time $T$.

The mean-shift formula~\eqref{eq:score-barycenter} expresses $s(x,t)$
as the average of $(y-x)/(2t)$ under $\pi_{x,t}$. Jensen's inequality, followed by
integration in $x$, gives
\begin{align*}
 \int_{\R^d}|s(x,t)|^2u(x,t)\,dx
 &\leq\frac1{4t^2}\iint_{\R^d\times\R^d}|x-y|^2G_t(x-y)\,dx\,du_0(y)\\
 &=\frac{d}{2t}.
\end{align*}
Using the push-forward identity and Cauchy--Schwarz, we find
\begin{align*}
 \int_{\R^d}\int_0^T
       |s(\Phi_{t,T}(x),t)|\,dt\,u(x,T)\,dx
 &=\int_0^T\int_{\R^d}|s(y,t)|u(y,t)\,dy\,dt\\
 &\leq\int_0^T\sqrt{\frac{d}{2t}}\,dt=\sqrt{2dT}.
\end{align*}
Thus for $u(T)$-almost every $x$ the trajectory has finite length:
\[
 \int_0^T|s(\Phi_{t,T}(x),t)|\,dt<\infty.
\]
For such an $x$, the ODE implies, for $0<a<b\leq T$,
\[
 |\Phi_{b,T}(x)-\Phi_{a,T}(x)|
 \leq\int_a^b|s(\Phi_{t,T}(x),t)|\,dt.
\]
The right-hand side tends to zero as $a,b\to0^+$, so the trajectory
has a finite limit, denoted by $\Phi_{0,T}(x)$.
The finite-length set is Borel. Define $\Phi_{0,T}$ to be a fixed
point of $S$ on its complement. The resulting map is Borel measurable,
as follows by taking the pointwise limit along any fixed sequence
$t_k\to0^+$ on the finite-length set.

To locate that limit, let $\varphi$ be bounded and continuous. Dominated
convergence and the heat-flow limit give
\begin{align*}
 \int_{\R^d}\varphi(\Phi_{0,T}(x))u(x,T)\,dx
 &=\lim_{t\to0^+}\int_{\R^d}\varphi(\Phi_{t,T}(x))u(x,T)\,dx\\
 &=\lim_{t\to0^+}\int_{\R^d}\varphi(y)u(y,t)\,dy
   =\int_{\R^d}\varphi\,du_0.
\end{align*}
Therefore $(\Phi_{0,T})_\#u(T)=u_0$, and
$\Phi_{0,T}(x)\in S$ for $u(T)$-almost every $x$.
Since $u(x,T)>0$ everywhere, $u(T)$ and Lebesgue measure have the same
null sets. This proves assertion~\emph{(i)}.

\par\medskip
\noindent\textbf{Step 2 (Convergence rate to the convex hull).} 
Recall that
\[
K =  \operatorname{conv}(S) = \operatorname{conv} (\operatorname{supp}(u_0)).
\]
Since $S$ is compact, the convex hull $K$ is also compact. Let
\[
V(x):=\frac12\,\operatorname{dist}(x,K)^2 \quad \text{and} \quad p(x):=\operatorname{proj}_K(x).
\]
Then $V\in C^1(\R^d)$ with
\[
\nabla V(x)=x-p(x).
\]
Along any trajectory of \eqref{eq:ODE_proof},
\begin{equation}\label{eq:chain}
\frac{d}{dt} V\big( X_t\big)
=\big\langle \nabla V\big( X_t\big),\,\dot{ X}_t\big\rangle
=\big\langle  X_t-p\big( X_t\big),\, -\nabla\log u\big( X_t,t\big)\big\rangle.
\end{equation}
By the Gaussian mean-shift formula \eqref{eq:score-barycenter},
\[
 s(x,t)=\frac{m(x,t)-x}{2t},
\]
where $m(x,t)$ is the barycenter in \eqref{eq:posterior-mean}.
Thus, using \eqref{eq:chain}, we have
\[
\frac{d}{dt} V\big(X_t\big)
= \frac{1}{2t}\,\big\langle X_t - m\big(X_t, t\big),\, X_t - p\big(X_t\big)\big\rangle.
\]
Since \(m(x,t)\) is a barycenter with positive weights supported on \(\operatorname{supp}(u_0)\), it follows that
\(m(x,t) \in K\) for all \((x,t) \in Q\).

\vspace{0.5em}
For any nonempty closed convex set \(A \subset \R^d\), and for any \(x \in \R^d\) and \(y \in A\), the following inequality holds:
\begin{equation*}
    \langle x - y,\; x - \operatorname{proj}_A(x) \rangle \;\ge\; |x - \operatorname{proj}_A(x)|^2.
\end{equation*}

\vspace{0.5em}
Applying this inequality with \(A = K\), \(y = m(X_t,t)\), and \(x = X_t\), we obtain
\[
\big\langle X_t - m\big(X_t,t\big),\, X_t - p\big(X_t\big)\big\rangle 
\;\ge\; \big| X_t - p\big(X_t\big) \big|^2.
\]
Therefore,
\[
\frac{d}{dt} V\big( X_t\big)\ \ge\ \,\frac{1}{2 t}\,| X_t-p( X_t)|^2
= \frac{1}{t}\,V\big( X_t\big).
\]
This implies that $\big(\,V(X_t)/t\,\big)' \ge 0$. Hence,
\[
\frac{V(X_t)}{t} \leq \frac{V(X_T)}{T}\quad \forall\, t\in(0,T).
\]
The desired estimate \eqref{eq:dist-bound} follows immediately from the definition of $V$.

\medskip
\noindent\textbf{Step 3 (Convergence rate for finite support).}
Let
\[
    u_0=\sum_{k=1}^N w_k\delta_{y_k},\qquad
    w_k=u_0(\{y_k\})>0,\qquad \sum_{k=1}^Nw_k=1.
\]
When $N=1$, the score formula gives
$X_t-y_1=\sqrt{t/T}(x_T-y_1)$, so the result is immediate.
Assume $N\geq2$ and fix a terminal point for which Step~1 applies.
Since the sites are distinct and the support is finite, the full limit
is a site $y_i$. We use this convergence only to enter the corresponding
Voronoi core.
Define
\begin{equation}\label{eq:gamma}
\gamma := \frac{1}{2}\min_{j\neq i}|y_i - y_j|^2 \;>\; 0,
\end{equation}
and introduce the corresponding \emph{$\gamma$-Voronoi core} of \(y_i\),
illustrated in Figure~\ref{fig:Voronoi}:
\[
V_i(\gamma)
:= \Big\{x\in\R^d : |x - y_j|^2 - |x - y_i|^2 \ge \gamma \;\; \forall\, j\neq i \Big\}.
\]

\begin{figure}[htbp]
  \centering
  \begin{subfigure}[b]{0.45\textwidth}
    \centering
    \includegraphics[width=\linewidth]{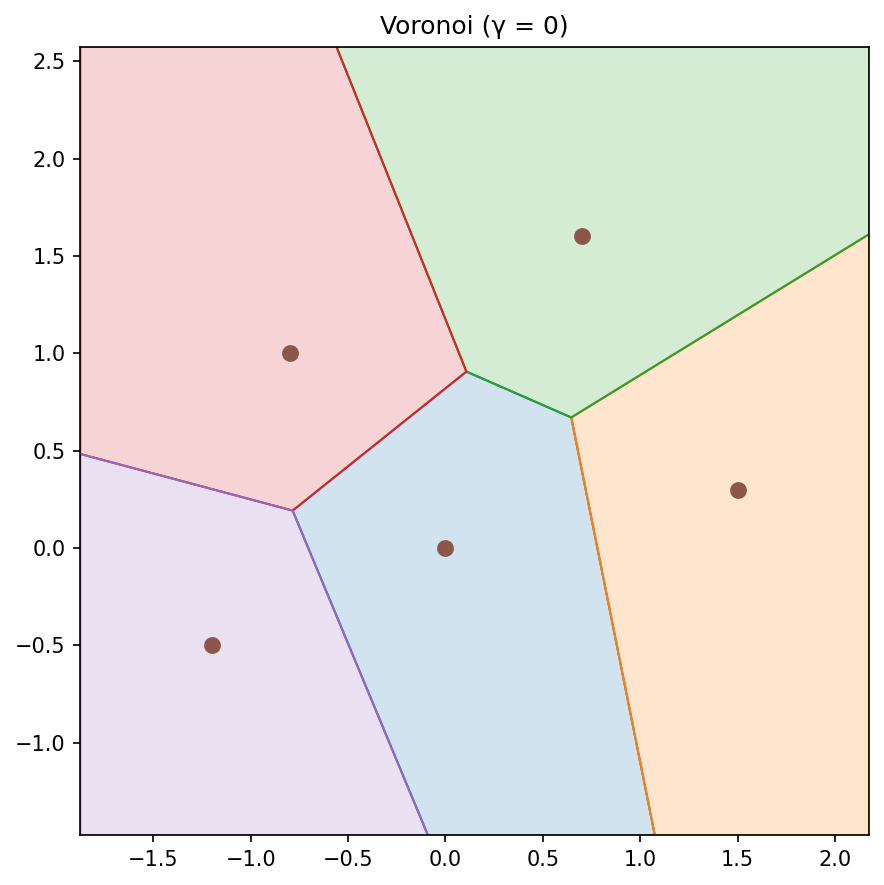}
    \caption{Voronoi diagram ($\gamma = 0$)}
  \end{subfigure}
  \hfill
  \begin{subfigure}[b]{0.45\textwidth}
    \centering
    \includegraphics[width=\linewidth]{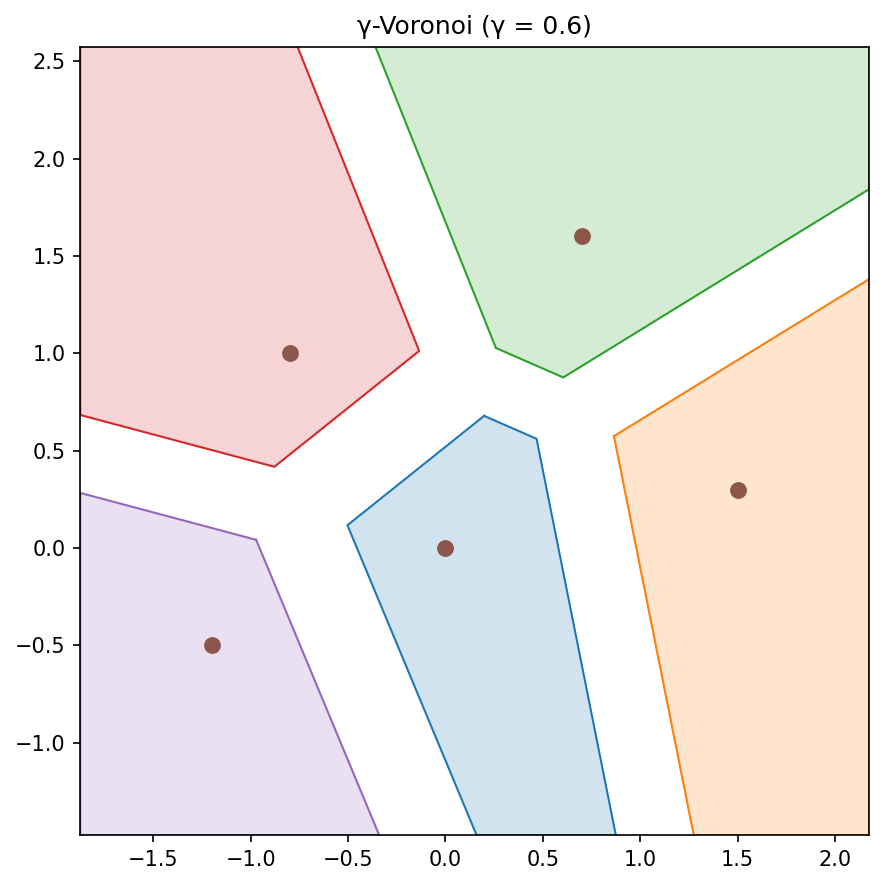}
    \caption{$\gamma$--Voronoi core ($\gamma = 0.6$)}
  \end{subfigure}

  \caption{Voronoi diagram and $\gamma$--Voronoi cores for five points in $\R^2$. When $\gamma=0$, we recover the classical Voronoi diagram.}
  \label{fig:Voronoi}
\end{figure}

\medskip
We use the following two claims, proved below:

\begin{enumerate}
\item\textbf{Claim 1 (Varadhan-type concentration for the Gaussian mean shift).} For any $x\in V_i(\gamma)$ and any $t>0$,
\[
\bigl|m(x,t)-y_i\bigr|
\;\le\;
\Bigl(\sum_{j\neq i}\frac{w_j}{w_i}\,|y_j-y_i|\Bigr)\,
e^{-\frac{\gamma}{4t}},
\]
where
\[
m(x,t)=\frac{\sum_{k=1}^N w_k\,e^{-\frac{|x-y_k|^2}{4 t}}\,y_k}
               {\sum_{k=1}^N w_k\,e^{-\frac{|x-y_k|^2}{4 t}}}.
\]
\item\textbf{Claim 2 (Invariance of the $\gamma$--Voronoi core at late generation times).} There exists $t_\ast\in(0,T)$ such that
\[
 X_t\in V_i(\gamma)\quad\text{for all }t\in(0, t_\ast).
\]
\end{enumerate}

\medskip
Assuming the two claims above, we proceed with the proof. Let \[ Z_t:= X_t-y_i.\] Using $\nabla\log u(x,t)=(m(x,t)-x)/ 2t$, we compute
\[
\frac{d}{dt}\,\frac12| Z_t|^2
=\frac{1}{2t}\big(\langle Z_t,\,y_i-m( X_t,t)\rangle + | Z_t|^2\big).
\]
By Claims~1-2, for all $t\in(0, t_\ast)$,
\[
\bigl|m( X_t,t)-y_i\bigr|
\;\le\;
\underbrace{\Bigl(\sum_{j\neq i}\frac{w_j}{w_i}\,|y_j-y_i|\Bigr)}_{=:C_i}\,
e^{-\frac{\gamma}{4t}}.
\]
Hence, for all $t\in(0, t_\ast)$,
\[
\frac{d}{dt}\,\frac12| Z_t|^2
\ge \frac{1}{2t}| Z_t|^2 - \frac{1}{2t}|Z_t|\,C_i\,e^{-\frac{\gamma}{4t}}.
\]
Introduce the similarity time variable \(\tau := -\log t\) (so that \(d\tau/dt = -1/t\)) and define \(\rho(\tau) := | Z_{t(\tau)} |\). 
The function $\rho$ is locally Lipschitz. The inequality below holds
almost everywhere, also at its zeros (where its derivative is zero
almost everywhere). Using the chain rule,
\[
\frac{d\rho}{d\tau} = \frac{d\rho}{dt}\frac{dt}{d\tau} = -\,t\,\frac{d\rho}{dt}.
\]
Substituting the inequality above, we obtain
\[
\frac{d\rho}{d\tau} \;\le\; -\frac{1}{2}\rho(\tau) + \frac{C_i}{2}\, e^{-\frac{\gamma}{4} e^{\tau}}, 
\qquad \tau \ge \tau_\ast := -\log t_\ast.
\]
Multiplying both sides by \(e^{\tau/2}\), we get
\[
\frac{d}{d\tau}\!\left(e^{\frac{\tau}{2}}\rho(\tau)\right) 
\;\le\; \frac{C_i}{2}\, e^{\frac{\tau}{2}} e^{-\frac{\gamma}{4} e^{\tau}}.
\]
Integrating this differential inequality from \(\tau_\ast\) to \(\tau \ge \tau_\ast\) yields
\[
e^{\frac{\tau}{2}}\rho(\tau) - e^{\frac{\tau_\ast}{2}}\rho(\tau_\ast) 
\;\le\; \frac{C_i}{2}\int_{\tau_\ast}^\tau e^{\frac{\sigma}{2}} e^{-\frac{\gamma}{4} e^{\sigma}}\, d\sigma.
\]
Thus,
\[
\rho(\tau) \;\le\; e^{-\frac{\tau}{2}} e^{\frac{\tau_\ast}{2}} \rho(\tau_\ast) 
+ \frac{C_i}{2} e^{-\frac{\tau}{2}}\int_{\tau_\ast}^\tau e^{\frac{\sigma}{2}} e^{-\frac{\gamma}{4} e^{\sigma}}\, d\sigma.
\]
Returning to the original variable \(t = e^{-\tau}\), we obtain the bound
\[
| Z_t | 
\;\le\; \left(\frac{t}{t_\ast}\right)^{\frac12}| Z_{t_\ast} |
+ \frac{C_i}{2} t^{\frac12} \int_{-\log t_\ast}^{-\log t} 
e^{\frac{\sigma}{2}} e^{-\frac{\gamma}{4} e^{\sigma}}\, d\sigma.
\]
The integral term is bounded by 
\[
\int_{-\log t_\ast}^{-\log t} 
e^{\frac{\sigma}{2}} e^{-\frac{\gamma}{4} e^{\sigma}}\, d\sigma \leq \int_{\R} e^{\frac{\sigma}{2}} e^{-\frac{\gamma}{4} e^{\sigma}}\, d\sigma = \frac{2\sqrt{\pi}}{\sqrt{\gamma}},
\]
where the last equality follows by a change of variables and the Gamma integral.
Therefore,
\[
| Z_t|\ \le\ \Big((t_\ast)^{-1/2}| Z_{t_\ast}| + C_i\sqrt{\tfrac{\pi}{\gamma}}\Big)\,t^{1/2},
\qquad \forall \,t\in(0,t_{\ast}).
\]
This is the required square-root rate near zero.
Here, the constants \(t_\ast\), \(C_i\), and \(\gamma\) also depend on the data sites and weights.
Since \(Z_t\) is continuous, the desired convergence rate~\eqref{eq:conv_1} holds over the entire interval \((0,T]\) upon choosing a sufficiently large constant \(C\).

\medskip
\begin{proof}[Proof of Claim 1]
Fix the index $i$ selected in Step~3 above.
For any $x\in \R^d$, let us define
\begin{equation}\label{eq:Lyapunov}
    \psi_j(x) \coloneqq |x-y_j|^2-|x-y_i|^2, \quad  \forall j\neq i.
\end{equation}
By the definition of $V_i(\gamma)$, we have that for any $x\in V_i(\gamma)$,
\[
\psi_j(x) \geq \gamma, \quad  \forall j\neq i.
\]
    Factor the $i$-th term in the numerator and the denominator of $m(x,t)$:
\[
m(x,t)
= \frac{w_i y_i+\sum_{j\neq i} w_j e^{-\frac{\psi_j(x)}{4t}}\,y_j}
       {w_i + \sum_{j\neq i} w_j e^{-\frac{\psi_j(x)}{4t}}}
= y_i \;+\;
\frac{\sum_{j\neq i} w_j e^{-\frac{\psi_j(x)}{4t}}(y_j-y_i)}
     {w_i + \sum_{j\neq i} w_j e^{-\frac{\psi_j(x)}{4t}}}.
\]
Since the denominator is at least $w_i$, we obtain
\[
|m(x,t)-y_i|
\le \frac{1}{w_i}\sum_{j\neq i} w_j e^{-\frac{\psi_j(x)}{4t}}\,|y_j-y_i|
\le \Bigl(\sum_{j\neq i}\frac{w_j}{w_i}\,|y_j-y_i|\Bigr) e^{-\frac{\gamma}{4t}},
\]
because $e^{-\psi_j(x)/(4t)}\le e^{-\gamma/(4t)}$ for all $j\neq i$.
\end{proof}

\begin{proof}[Proof of Claim 2]
Let $d_{ij}:=|y_i-y_j|$ and recall
\[
C_i=\sum_{j\neq i}\frac{w_j}{w_i}\,|y_j-y_i|.
\]
Since $\gamma>0$, there exists $\tau_0 \in (0,T)$ such that for all $t\in(0, \tau_0)$,
\begin{equation}\label{eq:C_i}
C_i\,e^{-\frac{\gamma}{4t}}\ \le\ \frac18\,d_{ij}\qquad\forall\,j\neq i.
\end{equation}
By Step~1, $X_t\to y_i$, and $y_i\in\operatorname{int}V_i(\gamma)$.
We may therefore choose $0<t_\ast<\tau_0$ such that
$X_{t_\ast}\in\operatorname{int}V_i(\gamma)$.

Suppose that $X_t$ leaves the core when time decreases from $t_\ast$.
Let $\tau_1\in(0,t_\ast)$ be the lower endpoint of the connected
component of $\{t:X_t\in\operatorname{int}V_i(\gamma)\}$ containing
$t_\ast$. By continuity, $X_{\tau_1}\in\partial V_i(\gamma)$.
For an active index $j\ne i$, we have
\begin{equation}\label{eq:boundary}
 \psi_j(X_{\tau_1})=\gamma,\qquad
 \left.\frac{d}{dt}\psi_j(X_t)\right|_{t=\tau_1}\geq0,
\end{equation}
since $\psi_j(X_t)>\gamma$ immediately to the right of $\tau_1$.

Differentiating along the trajectory gives
\[
\left.\frac{d}{dt}\psi_j(X_t)\right|_{t=\tau_1}
=\frac{1}{\tau_1}\,\big\langle y_i-y_j,\, X_{\tau_1} - m\big( X_{\tau_1},\tau_1\big)\big\rangle.
\]
Decompose the scalar product:
\begin{align*}
\big\langle y_i-y_j,\, X_{\tau_1} -m\big(X_{\tau_1},\tau_1\big)\big\rangle
&=\big\langle y_i-y_j,\,  y_i - m\big( X_{\tau_1},\tau_1\big)\big\rangle
  +\big\langle y_i-y_j,\, X_{\tau_1} - y_i\big\rangle\\
&=\underbrace{\big\langle y_i-y_j,\, y_i- m\big( X_{\tau_1},\tau_1\big) \big\rangle}_{\text{error}}
 \;+\;\underbrace{\tfrac12\big(\psi_j( X_{\tau_1}) - d_{ij}^2\big)}_{\text{geometric}}.
\end{align*}

Since $ X_{\tau_1}\in\partial V_i(\gamma)\subset V_i(\gamma)$, Claim~1 yields
\begin{equation}\label{eq:error_1}
\big|\big\langle y_i-y_j,\,  y_i -m\big( X_{\tau_1},\tau_1\big)\big\rangle\big|
\ \le\ d_{ij}\,C_i\,e^{-\frac{\gamma}{4\tau_1}}
\ \le\ \frac18\,d_{ij}^2,
\end{equation}
where the last inequality uses \eqref{eq:C_i} and $\tau_1< t_\ast< \tau_0$. 

For the geometric term, since $\psi_j( X_{\tau_1})=\gamma$,
\begin{equation}\label{eq:error_2}
\frac12\big(\psi_j( X_{\tau_1}) - d_{ij}^2\big)
=\frac12\big(\gamma- d_{ij}^2\big)\ \le\ -\frac14\,d_{ij}^2,
\end{equation}
where the final bound follows from \eqref{eq:gamma}. 

Combining \eqref{eq:error_1} and \eqref{eq:error_2} gives
\[
\left.\frac{d}{dt}\psi_j(X_t)\right|_{t=\tau_1}
\ \le\ \frac{1}{\tau_1}\left(\frac18\,d_{ij}^2-\frac14\,d_{ij}^2\right)
=-\frac{1}{8\tau_1}\,d_{ij}^2\ <\ 0.
\]
This contradicts the nonnegative derivative in \eqref{eq:boundary}. Hence, $ X_t\in V_i(\gamma)$ for all $t < t_\ast$.
\end{proof}

\subsection{Proofs of the entropy lemmas}
\label{sec:proof_lemmas}

We first establish the terminal entropy bounds. We then prove
entropy contraction and present the dissipation calculation under
additional integrability and decay assumptions. The justification
of this calculation under the stated hypotheses, including the
terminal passage, is given in Appendix~\ref{app:entropy-justification}.

\begin{proof}[Proof of Lemma~\ref{lm:absolutely-continuous}]
The Gaussian bounds in Lemma~\ref{lem:gaussian-bounds} imply
\[
    |\log u(x,T)|\leq C(1+|x|^2),
    \qquad x\in\R^d.
\]
Since $v_T$ has finite absolute entropy and second moment,
\begin{align*}
    \int_{\R^d}
        v_T\left|\log\frac{v_T}{u(T)}\right|\,dx
    \leq
    \int_{\R^d}v_T|\log v_T|\,dx
    +C\int_{\R^d}(1+|x|^2)v_T\,dx<\infty.
\end{align*}
Thus $\mathrm{KL}(v_T\|u(T))<\infty$.

For the reverse divergence, the Gaussian upper bound shows
that $u(T)$ has moments of every order. Under the additional
assumption $|\log v_T(x)|\leq C(1+|x|^q)$, we therefore have
\[
    \int_{\R^d}
        u(x,T)\left|\log\frac{u(x,T)}{v_T(x)}\right|\,dx
    \leq
    C'\int_{\R^d}(1+|x|^2+|x|^q)u(x,T)\,dx
    <\infty.
\]
This proves $\mathrm{KL}(u(T)\|v_T)<\infty$.
\end{proof}

\begin{proof}[Proof of Lemma~\ref{lem:KL}]
Write $H(t)=\mathrm{KL}(v(t)\|u(t))$.
We first establish contraction under the stated hypotheses,
and then explain the dissipation identity.

\medskip
\noindent\textbf{Step 1 (Entropy contraction).}

On every interval $[t_1,t_2]\subset(0,T]$, the two laws evolve
under the same mass-preserving Markov transition operator,
supplied by the well-posedness theory of
Section~\ref{sec:proof-wellposed}.
The output density ratio is a conditional average of the
input ratio. Jensen's inequality applied to $z\mapsto z\log z$
therefore gives
\[
    H(t_1)\leq H(t_2),
    \qquad 0<t_1\leq t_2\leq T.
\]
In particular,

\[
    \mathrm{KL}\bigl(v(t)\|u(t)\bigr)
    \leq
    \mathrm{KL}\bigl(v_T\|u(T)\bigr)<\infty,
\]
where finiteness follows from
Lemma~\ref{lm:absolutely-continuous}.

If $\mathrm{KL}(u(T)\|v_T)<\infty$, the same argument with
the two laws interchanged gives

\[
    \mathrm{KL}\bigl(u(t)\|v(t)\bigr)
    \leq
    \mathrm{KL}\bigl(u(T)\|v_T\bigr)<\infty.
\]

The polynomial bound in
Lemma~\ref{lm:absolutely-continuous} is a sufficient condition
for this terminal finiteness.

When $\epsilon=0$, the transition is induced by an invertible
characteristic map. The Jacobian factors cancel in the density
ratio, so a change of variables gives equality:
\[
    H(t_1)=H(t_2).
\]
The same conclusion holds for the reverse relative entropy
whenever it is finite. This proves the deterministic case.

\medskip
\noindent\textbf{Step 2 (Entropy dissipation identity).}

Assume now that $\epsilon>0$. Interior smoothness and positivity
follow from Proposition~\ref{prop:wellposed}.
For the calculation below, we additionally assume sufficient
integrability and decay of the displayed terms to justify
differentiation under the integral and integrations by parts
on compact time intervals contained in $(0,T)$.
These auxiliary assumptions are used only to exhibit the
diffusion--transport cancellation; they are not imposed in
Lemma~\ref{lem:KL}.

The calculation is the FP analogue of the heat-flow identity
in~\cite[Prop.~1.6]{klartag2025strong}. We derive it directly
for the present equation.

Set
\[
    L(x,t):=\log\frac{v(x,t)}{u(x,t)},
    \qquad
    H(t)=\int_{\R^d}v(x,t)L(x,t)\,dx.
\]
Differentiating in time and using
$\partial_tL=(\partial_tv)/v-(\partial_tu)/u$, we obtain
\[
    H'(t)
    =
    \int_{\R^d}
        \left(
            (L+1)\partial_t v
            -\frac vu\,\partial_t u
        \right)\,dx.
\]
Both densities satisfy the same backward FP equation:
\[
    \partial_t v
    =-\epsilon\Delta v+(1+\epsilon)\operatorname{div}(sv),
    \qquad
    \partial_t u
    =-\epsilon\Delta u+(1+\epsilon)\operatorname{div}(su).
\]
Consequently,
\begin{align*}
    H'(t)
    &=
    -\epsilon\int_{\R^d}(L+1)\Delta v\,dx
    +\epsilon\int_{\R^d}\frac vu\,\Delta u\,dx\\
    &\quad
    +(1+\epsilon)\int_{\R^d}(L+1)\operatorname{div}(sv)\,dx
    -(1+\epsilon)\int_{\R^d}\frac vu\,
                    \operatorname{div}(su)\,dx.
\end{align*}
We treat the diffusion and transport contributions separately.

\medskip
\noindent\textit{Diffusion contribution.}
Integration by parts gives
\begin{align*}
    \mathcal D
    &:=
    -\epsilon\int_{\R^d}(L+1)\Delta v\,dx
    +\epsilon\int_{\R^d}\frac vu\,\Delta u\,dx\\
    &=
    \epsilon\int_{\R^d}\langle\nabla L,\nabla v\rangle\,dx
    -\epsilon\int_{\R^d}
        \left\langle\nabla\!\left(\frac vu\right),\nabla u\right\rangle\,dx\\
    &=
    \epsilon\int_{\R^d}
        \left\langle\nabla v-\frac vu\nabla u,\nabla L\right\rangle\,dx\\
    &=
    \epsilon\int_{\R^d}v|\nabla L|^2\,dx,
\end{align*}
where we used
$\nabla(v/u)=(v/u)\nabla L$ and
$\nabla v-(v/u)\nabla u=v\nabla L$.

\medskip
\noindent\textit{Transport contribution.}
The same integrations by parts yield
\begin{align*}
    \mathcal T
    &:=
    (1+\epsilon)\int_{\R^d}(L+1)
                   \operatorname{div}(sv)\,dx
    -(1+\epsilon)\int_{\R^d}\frac vu\,
                   \operatorname{div}(su)\,dx\\
    &=
    -(1+\epsilon)\int_{\R^d}v\,\langle s,\nabla L\rangle\,dx
    +(1+\epsilon)\int_{\R^d}
        u\,\left\langle s,\nabla\!\left(\frac vu\right)\right\rangle\,dx\\
    &=0.
\end{align*}
Thus the common transport field does not contribute to
the entropy derivative.

Combining the two contributions, we conclude that
\[
    H'(t)
    =
    \epsilon\int_{\R^d}
        v(x,t)\left|
            \nabla\log\frac{v(x,t)}{u(x,t)}
        \right|^2\,dx.
\]
Integrating on an interior time interval gives
\[
    H(t_2)-H(t_1)
    =
    \epsilon\int_{t_1}^{t_2}\!\int_{\R^d}
        v\left|\nabla\log\frac vu\right|^2\,dx\,dt,
    \qquad 0<t_1<t_2<T.
\]
Interchanging $u$ and $v$ gives the corresponding
reverse-entropy identity.

This establishes the dissipation identities under the auxiliary
assumptions of Step~2.
Appendix~\ref{app:entropy-justification} proves the integrated
identity under the original hypotheses, using convex entropy
truncations and spatial cutoffs. The $L^1$ terminal traces permit
the passage to $t_2=T$ for each truncated entropy, after which
monotone convergence removes the truncation.
This gives \eqref{KL:integrated}, absolute continuity of $H$
on every $[\tau,T]$, and the almost-everywhere identities
\eqref{KL:identity-1}--\eqref{KL:identity-2}.

\end{proof}

\subsection{Proof of Proposition~\ref{prop:score-stability}}
\label{sec:proof-score}

\begin{proof}[Proof of Proposition~\ref{prop:score-stability}]
We first derive the estimate under the same auxiliary
integrability and decay assumptions used in the entropy
calculation above. The justification under the hypotheses of
the proposition is given in
Appendix~\ref{app:entropy-justification}.

Write
\[
    g=v_\theta,\qquad e=s_\theta-s,\qquad
    L=\log\frac ug,\qquad
    H_\theta(t)=\mathrm{KL}\bigl(u(t)\|g(t)\bigr).
\]
The equation for $g$ takes the form
\[
    \partial_t g
    =
    -\epsilon\Delta g
    +(1+\epsilon)\operatorname{div}(sg)
    +(1+\epsilon)\operatorname{div}(eg).
\]
The diffusion calculation in the proof of Lemma~\ref{lem:KL}
is unchanged, and the transport terms involving the common
field $s$ cancel. Only the additional drift $e$ remains:
\begin{align*}
    H_\theta'(t)
    &=
    \epsilon\int_{\R^d}u|\nabla L|^2\,dx
    -(1+\epsilon)\int_{\R^d}
        \frac ug\,\operatorname{div}(eg)\,dx\\
    &=
    \epsilon\int_{\R^d}u|\nabla L|^2\,dx
    +(1+\epsilon)\int_{\R^d}
        g \left\langle e,\nabla\!\left(\frac ug\right)\right\rangle\,dx\\
    &=
    \epsilon\int_{\R^d}u|\nabla L|^2\,dx
    +(1+\epsilon)\int_{\R^d}u \langle e,\nabla L\rangle\,dx.
\end{align*}
Since $\epsilon>0$, completing the square gives
\[
    \epsilon|\nabla L|^2+(1+\epsilon)\langle e,\nabla L\rangle
    =
    \epsilon\left|
        \nabla L+\frac{1+\epsilon}{2\epsilon}e
    \right|^2
    -\frac{(1+\epsilon)^2}{4\epsilon}|e|^2.
\]
Therefore,
\[
    H_\theta'(t)
    \geq
    -\frac{(1+\epsilon)^2}{4\epsilon}
        \int_{\R^d}u(x,t)|e(x,t)|^2\,dx.
\]
Integrating from $t$ to $t_2<T$, we obtain
\[
    H_\theta(t)
    \leq
    H_\theta(t_2)
    +\frac{(1+\epsilon)^2}{4\epsilon}
        \int_t^{t_2}\!\int_{\R^d}
            u(x,r)|e(x,r)|^2\,dx\,dr.
\]
Appendix~\ref{app:entropy-justification} establishes this
inequality under the original assumptions and permits the passage
$t_2\to T^-$. Substituting $e=s_\theta-s$ then gives
\eqref{eq:score-mismatch-bound}. It uses the same entropy truncations and spatial
cutoffs as for Lemma~\ref{lem:KL}, with an additional flux
arising from the score error. That flux is controlled by
the linear growth of $e$ and the masses of $u$ and $g$. 
\end{proof}

\section{Conclusions and perspectives}\label{sec:conclusion}
In this paper, we have studied the stability and the singular endpoint
of exact-score generative dynamics. The $L^p$ estimate controls individual
probability densities and differences of probability solutions, while
relative-entropy dissipation
gives a quantitative $W_2$ limit for both deterministic and stochastic
generation. At the trajectory level, finite mean length yields
almost-everywhere convergence to the data support. For empirical data,
the Voronoi estimate provides the corresponding square-root rate.

The explicit case $\epsilon=1$ further separates the two roles of the
reference score and the terminal law: the former fixes the support, and
the latter changes the endpoint weights. The expansion through order
$T^{-2}$ describes this change for Gaussian initialization and explains
the cancellation obtained by matching the barycenter. The score-error
estimate then relates stability to the loss used in training, without
requiring the perturbed endpoint to have the original support.

Several questions remain. For non-empirical supports, trajectory-wise
rates require information beyond the law-level estimate proved here.
The structure of exceptional deterministic initializations and the
small-time fluctuations of stochastic trajectories also deserve further
study. On the approximation side, a quantitative comparison with the
unknown population law must account for the learned score and the
numerical solver. The present estimates provide complementary
controls on initialization, score perturbation, and positive-time stopping
that may serve as ingredients in such an analysis.

\appendix

\section{Localization and approximation for the energy estimate}
\label{app:wellposed}\label{app:energy-justification}

We justify the energy calculation in Section~\ref{sec:proof-Lp}
under the hypotheses of Theorem~\ref{thm:energy_score}.
We first approximate the terminal density by compactly supported Lipschitz
probability densities and establish the energy estimate
for the corresponding solutions using spatial cutoffs.
Passing to the limit then gives the estimate and the strong
$L^p$ terminal trace for the original datum.

\begin{proof}[Proof of Theorem~\ref{thm:energy_score}]
Let $v_T$ be a probability density in $L^p(\R^d)$ and fix
$0<\tau<T$. By Lemma~\ref{lem:gaussian-bounds}, the score $s$
is smooth, globally Lipschitz in space, and of at most linear
growth, uniformly on $[\tau,T]$. Moreover,
\[
    \operatorname{div}s(x,t)\geq-\frac{d}{2t}.
\]

\medskip
\noindent\textbf{Step 1 (Approximation of the terminal density).}
Let $\operatorname{Lip}_c(\R^d)$ denote the space of compactly
supported Lipschitz functions. Choose nonnegative
$v_T^n\in\operatorname{Lip}_c(\R^d)$ of mass one such that
\[
    v_T^n\longrightarrow v_T
    \qquad\text{in }L^1(\R^d)\cap L^p(\R^d).
\]
We first establish the estimate for nonnegative terminal data
in $\operatorname{Lip}_c(\R^d)$, a class containing both
$v_T^n$ and $|v_T^n-v_T^m|$.

\medskip
\noindent\textbf{Step 2 (Spatial localization and terminal traces).}
Let $w_T\in\operatorname{Lip}_c(\R^d)$ be nonnegative.
The zero datum is immediate.
For $\epsilon>0$, the construction in
\cite[Prop.~5]{lebris2007fokker}, applied after the time change
in \eqref{eq:forward-FP}, gives a nonnegative solution
\[
    w\in L^\infty\bigl([\tau,T];
        L^1(\R^d)\cap L^\infty(\R^d)\bigr).
\]
The coefficient hypotheses follow from the constant diffusion
matrix and the preceding bounds on $s$ and its spatial derivative.
For $\epsilon=0$, the characteristic formula gives the same
bounded-integrable solution class.

Choose a nonnegative radial nonincreasing
$\chi\in C_c^\infty(\R^d)$, equal to one on $B_1$ and zero
outside $B_2$, and set
\[
    \chi_L(x)=\chi(x/L),\qquad
    A_L=\{x\in\R^d:L\leq|x|\leq2L\},\qquad L\geq1.
\]
Then $\chi_L\uparrow1$,
$|\nabla\chi_L|\leq C/L$, and $|\Delta\chi_L|\leq C/L^2$,
with both derivatives supported in $A_L$.
Testing the equation with $\chi_L$ and using its weak terminal
trace gives
\[
    \left|
        \int_{\R^d}\chi_Lw(t)\,dx
        -\int_{\R^d}\chi_Lw_T\,dx
    \right|
    \leq C_{\tau,\epsilon}
        \int_t^T\!\int_{A_L}w(x,r)\,dx\,dr.
\]
Linear growth of $s$ makes the constant independent of $L$.
The right-hand side tends to zero as $L\to\infty$, proving
conservation of mass.

Local continuity, including at $T$, follows from parabolic
regularity for the continuous terminal datum when $\epsilon>0$,
and from characteristics when $\epsilon=0$.
Nonnegativity and conservation of mass then give strong
$L^1$ continuity by Scheff\'e's lemma. Interpolation yields
\[
    w\in C\bigl([\tau,T];L^q(\R^d)\bigr),
    \qquad 1\leq q<\infty.
\]
By uniqueness after normalization, $w(t)=\mathcal S_{t,T}w_T$,
where $\mathcal S_{t,T}$ is the probability solution operator
of Proposition~\ref{prop:wellposed}, extended to finite
nonnegative data by scaling, with $\mathcal S_{t,T}0=0$.
This operator is positive and additive.

\medskip
\noindent\textbf{Step 3 (Energy estimate).}
Assume first that $p>1$ and define
\[
    \beta_\delta(z)
    =\frac{(z^2+\delta^2)^{p/2}-\delta^p}{p},
    \qquad
    A_\delta(z)=z\beta_\delta'(z)-\beta_\delta(z),
    \qquad \delta>0.
\]
Then $\beta_\delta''\geq0$ and $A_\delta\geq0$.
Set $E_{L,\delta}(r)=\int_{\R^d}\chi_L\beta_\delta(w(r))\,dx$.
Multiplying the equation by $\chi_L\beta_\delta'(w)$,
integrating by parts in space, and integrating from $t$ to
$t_2<T$ gives
\begin{align}
    E_{L,\delta}(t)
    &+\epsilon\int_t^{t_2}\!\int_{\R^d}
        \chi_L\beta_\delta''(w)|\nabla w|^2\,dx\,dr
    \nonumber\\
    &=
    E_{L,\delta}(t_2)
    -(1+\epsilon)\int_t^{t_2}\!\int_{\R^d}
        \chi_LA_\delta(w)\operatorname{div}s\,dx\,dr
    +\mathcal R_{L,\delta}(t,t_2),
    \label{eq:energy-identity-local}
\end{align}
where
\[
    \mathcal R_{L,\delta}(t,t_2)
    =
    \int_t^{t_2}\!\int_{\R^d}
        \beta_\delta(w)
        \left(\epsilon\Delta\chi_L
              +(1+\epsilon)\langle s,\nabla\chi_L\rangle\right)\,dx\,dr.
\]
For $\epsilon=0$, the same identity follows from the chain
rule along characteristics. This is the time-integrated,
localized form of \eqref{eq:energy_identity}.

Discard the nonnegative diffusion term and use
$\operatorname{div}s(x,r)\geq-d/(2r)$.
For $z\geq0$ and $0<\delta\leq1$,
\[
    0\leq\beta_\delta(z)+A_\delta(z)
    =z^2(z^2+\delta^2)^{(p-2)/2}
    \leq C_p(z+z^p),
\]
while $\beta_\delta(z)\to z^p/p$ and
$A_\delta(z)\to(p-1)z^p/p$ as $\delta\to0^+$.
The bounds from Step~2 therefore justify dominated convergence
as $\delta\to0^+$ in both endpoint terms
$E_{L,\delta}(t)$ and $E_{L,\delta}(t_2)$, and in the remaining
space--time integrals. 

Multiplying the resulting inequality by $p$ and then letting
$t_2\to T^-$, using the strong $L^p$ terminal trace, yields
\begin{equation}\label{eq:localized-Lp-bound}
    F_L(t)\leq F_L(T)
    +c_p\int_t^T\frac{F_L(r)}r\,dr+\mathcal R_L(t),
\end{equation}
where
\[
    F_L(t)=\int_{\R^d}\chi_Lw(t)^p\,dx,
    \qquad c_p=\frac{d(1+\epsilon)(p-1)}2,
\]
and
\[
    |\mathcal R_L(t)|
    \leq C_{\tau,\epsilon}
        \int_t^T\!\int_{A_L}w^p\,dx\,dr
    \longrightarrow0
    \qquad\text{as }L\to\infty.
\]
Here the cutoff derivatives compensate for the linear growth
of $s$, and the preliminary $L^1\cap L^\infty$ bound makes
$w^p$ integrable in space and time.

Monotone convergence in \eqref{eq:localized-Lp-bound} gives
\[
    F(t)\leq F(T)+c_p\int_t^T\frac{F(r)}r\,dr,
    \qquad F(t)=\|w(t)\|_{L^p}^p.
\]
The backward Gronwall argument of Section~\ref{sec:proof-Lp}
therefore yields
\begin{equation}\label{eq:energy-estimate-proof}
    \|w(t)\|_{L^p}
    \leq
    \left(\frac{T}{t}\right)^{c_p/p}\|w_T\|_{L^p},
    \qquad \tau\leq t\leq T.
\end{equation}
For $p=1$, nonnegativity and mass conservation give
$\|w(t)\|_{L^1}=\|w_T\|_{L^1}$, so the same estimate holds
with $c_1=0$.

\medskip
\noindent\textbf{Step 4 (Passage to the limit).}
Set $v^n(t)=\mathcal S_{t,T}v_T^n$.
Decomposing the terminal densities into their common part
$\min\{v_T^n,v_T^m\}$ and their nonnegative remainders,
positivity and additivity give
\begin{equation}\label{eq:positive-comparison-approx}
    |v^n(t)-v^m(t)|
    \leq\mathcal S_{t,T}|v_T^n-v_T^m|
    \qquad\text{a.e. in }\R^d.
\end{equation}
Since $|v_T^n-v_T^m|\in\operatorname{Lip}_c(\R^d)$,
Step~3 implies
\begin{align*}
    \sup_{t\in[\tau,T]}\|v^n(t)-v^m(t)\|_{L^p}
    &\leq
    \left(\frac{T}{\tau}\right)^{c_p/p}
        \|v_T^n-v_T^m\|_{L^p},\\
    \sup_{t\in[\tau,T]}\|v^n(t)-v^m(t)\|_{L^1}
    &\leq\|v_T^n-v_T^m\|_{L^1}.
\end{align*}
Thus $v^n$ converges in
$C([\tau,T];L^1)\cap C([\tau,T];L^p)$ to a nonnegative
unit-mass limit with terminal density $v_T$.
Passing to the limit against compactly supported test functions
shows that it solves \eqref{eq:FP}. By
Proposition~\ref{prop:wellposed}, this limit is the unique
weak probability solution $v$.

Passing to the limit in \eqref{eq:energy-estimate-proof}
proves \eqref{eq:energy_score}. The convergence in
$C([\tau,T];L^p)$ gives the strong terminal trace.
Since $\tau>0$ is arbitrary, the conclusion holds on $(0,T]$.
\end{proof}

\section{Entropy truncation and terminal passages}
\label{app:entropy-justification}

We justify the entropy calculations of
Sections~\ref{sec:proof_lemmas} and~\ref{sec:proof-score}
under the hypotheses of Lemma~\ref{lem:KL} and
Proposition~\ref{prop:score-stability}.
For $\epsilon>0$, interior smoothness and positivity follow
from the positive-time theory. What remains is to justify
spatial integration and the terminal passage without the
additional integrability and decay assumptions used in
the main text.

We first replace the entropy by convex truncations of at
most linear growth and localize in space using the cutoffs
$\chi_L$ and annuli $A_L$ from
Appendix~\ref{app:wellposed}, Step~2.
The truncated entropy densities are bounded by a constant
times the sum of the two probability densities, so their
integrability is known before any entropy estimate is used.
We then remove the spatial cutoff, pass to the terminal time
using the $L^1$ traces, and finally remove the entropy
truncation by monotone convergence.
This proves the integrated identities and the score-error
estimate, including their terminal values.

\begin{proof}[Proof of Lemma~\ref{lem:KL}]
We prove the integrated identity directly from the equation. The
entropy is truncated before integration, so its finiteness at
intermediate times is not assumed.

\medskip
\noindent\textbf{Step 1 (The deterministic case).}
If $\epsilon=0$, both laws are transported by the same invertible
characteristic map on every interval $[t_1,t_2]\subset(0,T]$.
The Jacobian factors cancel in the density ratio. A change of variables
therefore gives
\[
 \mathrm{KL}\bigl(v(t_1)\|u(t_1)\bigr)
 =\mathrm{KL}\bigl(v(t_2)\|u(t_2)\bigr).
\]
The same argument applies with $u$ and $v$ interchanged.

\medskip
\noindent\textbf{Step 2 (A truncated entropy identity).}
Assume $\epsilon>0$ and fix $0<t_1<t_2<T$. Both densities are smooth
and positive on this time interval. Set $h=v/u$. Subtracting the
equation for $u$, multiplied by $h$, from that for $v$ gives
\begin{equation}\label{eq:h-equation}
 \partial_t h+\epsilon\Delta h+(\epsilon-1)\langle s,\nabla h\rangle=0.
\end{equation}
Let $\theta_n\in C_c^\infty((0,\infty))$ satisfy
$0\leq\theta_n\uparrow1$, and define
\begin{equation}\label{eq:entropy-truncation}
 \Psi_n''(z)=\frac{\theta_n(z)}z,\qquad
 \Psi_n(1)=\Psi_n'(1)=0.
\end{equation}
Then $0\leq\Psi_n\uparrow\Psi$, where $\Psi(z)=z\log z-z+1$, and
$\Psi_n(z)\leq C_n(1+z)$. In addition, $\Psi_n'$ and
$\Psi_n-z\Psi_n'$ are bounded. Put
$\mathscr H_n=u\Psi_n(h)$ and $H_n(t)=\int_{\R^d}\mathscr H_n(x,t)\,dx$.
The product rule and \eqref{eq:h-equation} yield
\begin{equation}\label{eq:entropy-local-balance}
 \partial_t\mathscr H_n+\epsilon\Delta\mathscr H_n
       -(1+\epsilon)\operatorname{div}(s\mathscr H_n)
 =\epsilon u\Psi_n''(h)|\nabla h|^2.
\end{equation}
Notice that $0\leq\mathscr H_n\leq C_n(u+v)$, so this density is
integrable in space and time on $[t_1,t_2]$.

Multiply \eqref{eq:entropy-local-balance} by the spatial cutoff
$\chi_L$ from Appendix~\ref{app:wellposed}, Step~2, and integrate. We obtain
\begin{align*}
 \int_{\R^d}\chi_L\mathscr H_n(t_2)\,dx-\int_{\R^d}\chi_L\mathscr H_n(t_1)\,dx
 &=\epsilon\int_{t_1}^{t_2}\!\int_{\R^d}
                \chi_Lu\Psi_n''(h)|\nabla h|^2\,dx\,dt+\mathcal E_{n,L},\\
 |\mathcal E_{n,L}|
 &\leq C_{t_1,\epsilon}
              \int_{t_1}^{t_2}\!\int_{A_L}\mathscr H_n\,dx\,dt\longrightarrow0
              \qquad(L\to\infty).
\end{align*}
The estimate uses $|s(x,t)|\leq C_{t_1}(1+|x|)$ and
$|\nabla\chi_L|\leq C/L$ on $A_L$. Since $\chi_L\uparrow1$,
monotone convergence on the nonnegative right-hand side gives
\begin{equation}\label{eq:truncated-entropy-identity}
 H_n(t_2)-H_n(t_1)
 =\epsilon\int_{t_1}^{t_2}\!\int_{\R^d}
                     u\Psi_n''(h)|\nabla h|^2\,dx\,dt.
\end{equation}

\medskip
\noindent\textbf{Step 3 (Terminal trace and removal of truncation).}
For fixed $n$, the perspective function
$(f,g)\mapsto g\Psi_n(f/g)$ extends continuously to the nonnegative
quadrant and is Lipschitz there: its partial derivatives in the interior
are $\Psi_n'(f/g)$ and $\Psi_n(f/g)-(f/g)\Psi_n'(f/g)$, which are
bounded. The $L^1$ terminal traces of $u$ and $v$ therefore imply
$H_n(t_2)\to H_n(T)$ as $t_2\to T^-$. Thus
\eqref{eq:truncated-entropy-identity} holds also for $t_2=T$.

In particular, $H_n(t)\leq H_n(T)\leq H(T)$, where
$H(T)=\mathrm{KL}(v_T\|u(T))<\infty$. Letting $n\to\infty$ shows
that $H(t)=\int_{\R^d}u\Psi(v/u)\,dx=\mathrm{KL}(v(t)\|u(t))$ is finite and
bounded by $H(T)$. Since
\[
 u\Psi_n''(h)|\nabla h|^2
 \uparrow\frac{u}{h}|\nabla h|^2
 =v\left|\nabla\log\frac vu\right|^2,
\]
monotone convergence in \eqref{eq:truncated-entropy-identity} proves
\eqref{KL:integrated} for $0<t_1<t_2\leq T$. This identity gives the
absolute continuity of $H$ on every $[\tau,T]$ and
\eqref{KL:identity-1}. The local balance \eqref{eq:entropy-local-balance} depends only
on the fact that the two densities solve the same FP equation.
Repeating it with $v\Psi_n(u/v)$, followed by the same terminal
passage, proves \eqref{KL:identity-2} whenever
$\mathrm{KL}(u(T)\|v_T)<\infty$.
\end{proof}

\medskip
\begin{proof}[Proof of Proposition~\ref{prop:score-stability}]
The coefficient assumptions on $s_\theta$ give the same positive-time
probability construction as in Section~\ref{sec:proof-wellposed}.
The mass-conservation and $L^1$ approximation arguments of
Appendix~\ref{app:wellposed}, Steps~2 and~4, give the $L^1$ terminal
trace; these arguments do not use the Li--Yau bound.
Set
\[
    f=u,\qquad g=v_\theta,\qquad h=f/g,\qquad e=s_\theta-s,
\]
and use the same truncations $\Psi_n$ as above.
Both densities are smooth and positive for $0<t<T$, and
\[
    \partial_t g+\epsilon\Delta g
    -(1+\epsilon)\operatorname{div}(sg)
    =(1+\epsilon)\operatorname{div}(eg),
\]
whereas the left-hand operator annihilates $f$.
For $\mathscr H_n=g\Psi_n(h)$, the product and chain rules give
\begin{align}
    &\partial_t\mathscr H_n+\epsilon\Delta\mathscr H_n
      -(1+\epsilon)\operatorname{div}(s\mathscr H_n)
    \nonumber\\
    &\qquad=\epsilon g\Psi_n''(h)|\nabla h|^2
      +(1+\epsilon)gh\Psi_n''(h)\langle e,\nabla h\rangle
      +\operatorname{div}J_n,
    \label{eq:perturbed-entropy-local}\\
    &J_n=-(1+\epsilon)eg\bigl(h\Psi_n'(h)-\Psi_n(h)\bigr).
    \nonumber
\end{align}
Completing the square and using $h\Psi_n''(h)=\theta_n(h)\leq1$
yield
\[
    \epsilon g\Psi_n''(h)|\nabla h|^2
    +(1+\epsilon)gh\Psi_n''(h)\langle e,\nabla h\rangle
    \geq-\frac{(1+\epsilon)^2}{4\epsilon}f|e|^2.
\]
The additional cutoff term comes from $J_n$. Since
\[
    \left|h\Psi_n'(h)-\Psi_n(h)\right|
    =\left|\int_1^h\theta_n(z)\,dz\right|\leq|h-1|,
\]
linear growth of $e$ gives
$|J_n|/(1+|x|)\leq C(f+g)$ on every positive-time strip.
Together with $0\leq\mathscr H_n\leq C_n(f+g)$, this justifies
removing the spatial cutoff as in the preceding proof.
Writing $H_n(t)=\int_{\R^d}\mathscr H_n(x,t)\,dx$, we obtain
\[
    H_n(t)\leq H_n(t_2)
    +\frac{(1+\epsilon)^2}{4\epsilon}
       \int_t^{t_2}\!\int_{\R^d}u(x,r)|e(x,r)|^2\,dx\,dr,
    \qquad 0<t<t_2<T.
\]
The forcing is integrable on $[t,T]$ because $e$ has at most
linear growth and $u$ has uniformly bounded second moments there.
For fixed $n$, the Lipschitz property of the truncated perspective
and the $L^1$ terminal traces allow $t_2\to T^-$.
Letting then $n\to\infty$ by monotone convergence gives
\eqref{eq:score-mismatch-bound}.
\end{proof}

\section{Proofs for the explicit case}
\label{app:explicit}\label{sec:proof-explicit}

\begin{proof}[Proof of Proposition~\ref{prop:explicit-endpoint}]
The special structure at $\epsilon=1$ is seen by writing $v=uh$.
Since $\partial_tu=\Delta u$ and $su=\nabla u$,
\[
 \partial_t(uh)+\Delta(uh)-2\operatorname{div}(suh)
 =u(\partial_t h+\Delta h).
\]
Hence the quotient satisfies
\begin{equation}\label{eq:explicit-quotient}
    \partial_t h+\Delta h=0.
\end{equation}
This suggests
\[
    h(x,t)=\int_{\R^d}
       G_{T-t}(x-y)\frac{1}{u(y,T)}\,dv_T(y).
\]
Define $\widetilde v=uh$ by this formula. On each compact subset of
$\mathring Q$, the Gaussian bounds of Lemma~\ref{lem:gaussian-bounds}
bound the integrand and its derivatives by
$C(1+|y|)^k e^{-c|y|^2+C|y|}$ for suitable $c>0$, $C>0$, and $k$.
This is bounded in $y$ and integrable against $v_T$.
Differentiation under the integral is therefore justified, and
$\widetilde v$ solves \eqref{eq:FP} on $\mathring Q$.

It remains to identify the two endpoint traces. Introduce the
probability measure
\begin{equation}\label{eq:bridge-joint-law}
 d\Gamma_T(z,y)
 :=
 \frac{G_T(y-z)}{u(y,T)}\,du_0(z)\,dv_T(y).
\end{equation}
Indeed, integration in $z$ gives
\[
    \int_{\R^d}
    \frac{G_T(y-z)}{u(y,T)}\,du_0(z)=1.
\]
Its $y$-marginal is therefore $v_T$, while its $z$-marginal is
$a_{v_T}u_0$.

Set
\[
    c_t(z,y)
    =\left(1-\frac{t}{T}\right)z+\frac{t}{T}y,
    \qquad
    \kappa_t=t\left(1-\frac{t}{T}\right).
\]
Completing the square gives the Gaussian bridge identity
\begin{equation}\label{eq:bridge-factorization}
 G_t(x-z)G_{T-t}(y-x)
 =
 G_T(y-z)
 G_{\kappa_t}\bigl(x-c_t(z,y)\bigr).
\end{equation}
Using $u(t)=G_t*u_0$ in the definition of $\widetilde v$, we obtain
\begin{equation}\label{eq:explicit-bridge-representation}
 \widetilde v(x,t)
 =
 \iint_{\R^d\times\R^d}
 G_{\kappa_t}\bigl(x-c_t(z,y)\bigr)\,d\Gamma_T(z,y).
\end{equation}

This representation defines a narrowly continuous curve of probability
measures on $(0,T)$. Let $\varphi\in C_b(\R^d)$. Since the inner Gaussian in
\eqref{eq:explicit-bridge-representation} is a probability density,
dominated convergence gives
\[
 \lim_{t\to T^-}\int_{\R^d}\varphi\,d\widetilde v(t)
 =\iint_{\R^d\times\R^d}\varphi(y)\,d\Gamma_T(z,y)
 =\int_{\R^d}\varphi\,dv_T,
\]
because $\kappa_t\to0$ and $c_t(z,y)\to y$. Thus
the candidate $\widetilde v$ has terminal law $v_T$. Uniqueness in
Proposition~\ref{prop:wellposed} gives $\widetilde v=v$, proving
\eqref{eq:explicit-solution}.

Similarly, as $t\to0^+$,
$\kappa_t\to0$ and $c_t(z,y)\to z$, so
\[
 \lim_{t\to0^+}\int_{\R^d}\varphi\,dv(t)
 =
 \iint_{\R^d\times\R^d}\varphi(z)\,d\Gamma_T(z,y)
 =
 \int_{\R^d}\varphi(z)a_{v_T}(z)\,du_0(z).
\]
Theorem~\ref{thm:support} already gives convergence of $v(t)$ in
$W_2$ to the unique limiting law $v_0^*$. Hence the preceding narrow
limit identifies
\[
    dv_0^*=a_{v_T}\,du_0.
\]
Finally, Tonelli's theorem gives
\[
 \int_{\R^d}a_{v_T}(z)\,du_0(z)
 =
 \int_{\R^d}
 \frac{\int_{\R^d}G_T(y-z)\,du_0(z)}
      {u(y,T)}\,dv_T(y)
 =1.
\]
Since the kernel in the definition of $a_{v_T}$ is strictly positive,
we have $0<a_{v_T}<\infty$ $u_0$-almost everywhere. Thus $v_0^*$ and
$u_0$ are equivalent.
\end{proof}

\begin{proof}[Proof of Proposition~\ref{prop:moment-expansion}]
The proof follows the moment-expansion approach
of~\cite{duoandikoetxea1992moments}, applied to the rescaled
Gaussian convolution. Centering removes the first-moment term,
while Gaussian averaging determines the corrections that survive
in the endpoint weights. Since the explicit formula involves
the reciprocal of the heat density, we also control the
normalized denominator and the integrated reciprocal remainder.

\noindent
\textbf{Step 1 (Centering and rescaling).}
Consider first $v_T=G_T$. Set
\[
    \nu=(x\mapsto x-m_0)_\#u_0,\qquad
    \delta=(2T)^{-1/2},\qquad
    \varphi=G_{1/2}.
\]
Thus $\int_{\R^d} z\,d\nu(z)=0$ and
$\int_{\R^d} zz^\top\,d\nu(z)=\Sigma_0$.
Define
\[
    D_\delta(\xi)
    =
    \int_{\R^d}
        e^{\delta\langle\xi,z\rangle-\delta^2|z|^2/2}
        \,d\nu(z).
\]
The substitution $y=m_0+\xi/\delta$ in
\eqref{eq:explicit-endpoint}, followed by completion of the
square, gives
\begin{equation}\label{eq:moment-normalization}
    a_T(x)
    =
    e^{-\delta^2c}
    \int_{\R^d}
        \frac{\varphi(\xi)}{D_\delta(\xi+\delta b)}\,d\xi,
    \qquad
    b=x-2m_0,\quad c=\langle m_0,x-m_0\rangle.
\end{equation}
Since $S$ is compact, $b$ and $c$ remain bounded for $x\in S$.
All remainder estimates below are uniform on $S$.

\medskip
\noindent\textbf{Step 2 (Gaussian moments of the denominator).}
Put $k_\delta(\xi)=D_\delta(\xi+\delta b)-1$ and
\[
    H(z,w)=\langle z,w\rangle+\langle b,z+w\rangle.
\]
The Gaussian identity
\[
    \int_{\R^d}
        e^{\delta\langle\xi,z\rangle}\varphi(\xi)\,d\xi
    =e^{\delta^2|z|^2/2}
\]
allows us to integrate in $\xi$ before expanding.
Since $\nu$ is centered,
\[
    \iint_{\R^d\times\R^d} H\,d\nu(z)\,d\nu(w)=0,\qquad
    \iint_{\R^d\times\R^d} H^2\,d\nu(z)\,d\nu(w)
    =
    \operatorname{tr}(\Sigma_0^2)
    +2\langle\Sigma_0b,b\rangle.
\]
Consequently, Taylor's formula gives
\begin{align}
    \int_{\R^d}k_\delta\varphi\,d\xi
    &=
    \int_{\R^d}
        \left(e^{\delta^2\langle b,z\rangle}-1\right)d\nu(z)
    \nonumber\\
    &=
    \frac{\delta^4}{2}\langle\Sigma_0b,b\rangle
    +\mathcal O(\delta^6),
    \label{eq:moment-first}\\
    \int_{\R^d}k_\delta^2\varphi\,d\xi
    &=
    \iint_{\R^d\times\R^d} e^{\delta^2H(z,w)}\,d\nu(z)\,d\nu(w)
    -2\int_{\R^d} e^{\delta^2\langle b,z\rangle}\,d\nu(z)+1
    \nonumber\\
    &=
    \frac{\delta^4}{2}\operatorname{tr}(\Sigma_0^2)
    +\mathcal O(\delta^6).
    \label{eq:moment-second}
\end{align}
The remainders are uniform because $z,w$, and $b$ remain
in bounded sets.

\medskip
\noindent\textbf{Step 3 (Reciprocal expansion and remainder).}
Jensen's inequality and the centering of $\nu$ give
\begin{equation}\label{eq:moment-denominator}
    D_\delta(\xi+\delta b)
    \geq
    e^{-\delta^2\operatorname{tr}(\Sigma_0)/2}.
\end{equation}
Furthermore, Taylor's formula, with the linear term in
$\delta$ canceled by $\int_{\R^d} z\,d\nu(z)=0$, yields
\[
    |k_\delta(\xi)|
    \leq
    C\delta^2(1+|\xi|^2)
        e^{C\delta|\xi|+C\delta^2}.
\]
Gaussian integrability therefore implies
\begin{equation}\label{eq:moment-L3}
    \int_{\R^d}|k_\delta(\xi)|^3\varphi(\xi)\,d\xi
    =\mathcal O(\delta^6).
\end{equation}
Using the exact identity
\[
    \frac1{1+k}=1-k+k^2-\frac{k^3}{1+k},
\]
the lower bound \eqref{eq:moment-denominator} controls the
integrated remainder by a constant times
\eqref{eq:moment-L3}. Substituting
\eqref{eq:moment-first}--\eqref{eq:moment-second} into
\eqref{eq:moment-normalization} thus gives
\begin{align*}
    a_T(x)
    &=
    e^{-\delta^2c}
    \left(
        1+\frac{\delta^4}{2}
        \bigl(
            \operatorname{tr}(\Sigma_0^2)
            -\langle\Sigma_0b,b\rangle
        \bigr)
        +\mathcal O(\delta^6)
    \right)\\
    &=
    1-\delta^2c
    +\frac{\delta^4}{2}
        \bigl(
            c^2+\operatorname{tr}(\Sigma_0^2)
            -\langle\Sigma_0b,b\rangle
        \bigr)
    +\mathcal O(\delta^6).
\end{align*}
Since $\delta^2=(2T)^{-1}$, this proves
\eqref{eq:moment-expansion}--\eqref{eq:moment-coefficient}.

For $v_T=G_T(\cdot-m_0)$, translate the data and the terminal
Gaussian by $-m_0$. The translated data have zero mean and
covariance $\Sigma_0$. Applying the preceding expansion at
$x-m_0$, with the mean-mismatch factor absent, gives
\eqref{eq:centered-moment-expansion}.
\end{proof}

\end{document}